\documentclass[11pt]{amsart}

\issueinfo{00}{0}{Month}{Year}
\copyrightinfo{2008}{University of Calgary}
\pagespan{1}{2}
\PII{ISSN 1715-0868}
\include{cdmstdcmds}

\newtheorem{theorem}{Theorem}[section]
\newtheorem{lemma}[theorem]{Lemma}
\newtheorem{observation}[theorem]{Observation}

\theoremstyle{remark}

\numberwithin{equation}{section}
\usepackage{enumitem}
\setlist[itemize,1]{label=$\bullet$}
\usepackage{esvect}

\usepackage{siunitx}
\begin{document}

\title{On the Directed Hamilton-Waterloo Problem with Two Cycle Sizes}

\author{FAT\.{I}H YETG\.{I}N}
\address{Department of Mathematics, Gebze Technical University, Kocaeli, 41400, Turkey}
\email{fyetgin@gtu.edu.tr}

\author{U\u{g}ur Odaba\c{s}{\i}}
\address{Department of Engineering Sciences, Istanbul University-Cerrahpasa, Istanbul, 34320, Turkey}
\email{ugur.odabasi@iuc.edu.tr}

\author{S\.{I}Bel \"OZKAN}
\address{Department of Mathematics, Gebze Technical University, Kocaeli, 41400, Turkey}
\email{s.ozkan@gtu.edu.tr}

\date{(date1), and in revised form (date2).}


\begin{abstract}
The Directed Hamilton-Waterloo Problem asks for a directed 2-factorization of the complete symmetric digraph $K_v^*$ where there are two non-isomorphic 2-factors. In the uniform version of the problem, factors consist of either directed $m$-cycles or $n$-cycles. In this paper, necessary conditions for a solution of this problem are given, and the problem is solved for the factors with $(m, n)\in \{(4,6),(4,8),(4,12),(4,16),\allowbreak(6,12),(8,16), (3,5),(3,15),(5,15)\}$.
\end{abstract}

\maketitle

\section{Introduction} \label{S:intro}
A decomposition of a graph $G$ is a set $\mathcal{H}=\{H_{1}, H_{2}, \ldots, H_{k}\}$ of subgraphs of $G$ such that $\bigcup_{i=1}^{k} E\left(H_{i}\right)=E(G)$ and $E(H_{i}) \cap E(H_{j})=\emptyset$  for $ i \neq j$. Such a decomposition is called an $\{H_{1}, H_{2}, \ldots, H_{k}\}$-decomposition of $G$. A factor in a graph $G$ is a spanning (not necessarily connected) subgraph of $G$. If a graph $G$ can be decomposed into $r_{i}$ factors isomorphic to the factor $F_{i}$ for $i\in [1, t]$, then we say that $G$ has a $\left\{F_{1}^{r_{1}}, F_{2}^{r_{2}}, \ldots, F_{t}^{r_{t}}\right\}$-factorization. When each $F_{i}$ factor consists of only $n_{i}$ cycles for $i\in [1, t]$, then we will call the $F_{i}$ factor as a $C_{n_{i}}$-factor and call this factorization as a $\left\{C_{n_{1}}^{r_{1}}, C_{n_{2}}^{r_{2}}, \ldots, C_{n_{t}}^{r_{t}}\right\}$-factorization where each $r_i$ is the number of $C_{n_{i}}$-factors.

Graph factorizations constitute an important part of graph decomposition problems, especially when each factor is of regular degree. A $k$-regular spanning subgraph of $G$ is called a $k$-factor of $G$. It is easy to see that a 1-factor is a perfect matching in a graph and a 2-factor is either an Hamilton cycle or union of cycles. When it comes to 2-factorizations, there are two well-known graph factorization problems. One problem is the Oberwolfach Problem, which is posed by Ringel (see \cite{Ringel1971}) as a seating arrangement problem at a meeting in Oberwolfach. Given a conference venue with $k_{i}$ round tables, each of which has $m_{i}$ seats for $i \in [1, t]$, it asks whether it is possible that each participant of the conference (say $v$ many for odd $v$) sits next to (left or right) each other participant exactly once at the end of $\frac{v-1}{2}$ nights. In graph theory language, it asks whether the complete graph $K_v$ (or $K_v-I$ in the spouse avoiding version for even $v$) decomposes into isomorphic $2$-factors where each $2$-factor consists of $k_{i}$ $m_{i}$-cycles for each $i \in [1, t]$. This problem is denoted by OP$(m_{1}^{k_{1}}, m_{2}^{k_{2}}, \ldots, m_{t}^{k_{t}})$. If there is only one type of cycle, say of length $m$, in the factor, it can be denoted as OP$(m^{k})$, and its solution gives a $\{C_{m}^{\frac{v-1}{2}}\}$-factorization (or in short, a $C_m$-factorization) of $K_v$.

The Hamilton-Waterloo Problem is a generalization of the Oberwolfach Problem where there are two conference venues (one in Hamilton and one in Waterloo as one may guess) with different seating arrangements. This time each $2$-factor can be isomorphic to one of the given two 2-factors, say $F_1$ or $F_2$. If $F_1$ consists of only $m$-cycles and $F_2$ consists of only $n$-cycles, then the corresponding Hamilton-Waterloo Problem is called as the uniform version, and it is denoted by HWP$(v; m^{r}, n^{s})$ where $r$ and $s$ are the number of $C_m$ and $C_n$-factors where $r+s= \frac{v-1}{2}$, respectively. Having a solution to HWP$(v; m^{r}, n^{s})$ means that $K_v$ has a $\left\{C_{m}^{r}, C_{n}^{s}\right\}$-factorization for all possible $r$ and $s$ in the range.

The uniform versions of both problems are well-studied. In articles \cite{Alspach1985,Alspach1989,Hoffman1991}, authors solved completely the uniform version of Oberwolfach Problem. But the general case of the Oberwolfach Problem is still open. It is known that OP$\left(3^{2}\right)$, OP$\left(3^{4}\right)$, OP$(4,5)$ and OP$\left(3^{2}, 5\right)$ have no solution. In \cite{Adams2006,Deza2010}, it is shown that OP$(m_{1}^{k_{1}}, m_{2}^{k_{2}}, \ldots, m_{t}^{k_{t}})$ has a solution for all $n \leq 40$ with the above exceptions.

As the first results on the uniform Hamilton-Waterloo Problem, Adams et al.\cite{Adams2002} showed that HWP$(v; m^{r},n^{s})$ has a solution for all $v\leq 16$ and gave solutions for the small cases where $(m, n)\in \{(4,6),(4,8)$,$(4,16),(8,16),(3,5),$ $(3,15),(5,15)\}$. Cycle sizes $(3,4)$ and in general $(4,m)$ for odd $m$ has been studied by several authors (see \cite{Bonvicini2017}, \cite{Danziger2009}, \cite{Keranen}, \cite{Odabas2016}).

When $m$ and $n$ are odd, problem is almost completely solved in \cite{Burgess2017, Burgess2018} for odd $v$. In \cite{D.Bryant2011}, the problem is solved in the case of both $m$ and $n$ are even and $v \equiv 0 $ $(\bmod 4)$ except possibly when $r=s=1$. When $m$ and $n$ are both even and $v \equiv 2$ $(\bmod 4)$, this problem is solved by R. Haggkvist in \cite{Haggkvist1985} whenever $r$ and $s$ are both even. Also, if $m$ is even and $m\vert n$, the problem is completely solved in \cite{D.Bryant2013}.

One generalization of these problems may be to consider sitting on the right and sitting on the left of a participant as separate entities. To represent such a sitting, one has to use directed cycles which led us to work on directed graphs. There are studies on the directed Oberwolfach Problem, and here we work on the directed version of the Hamilton-Waterloo Problem. 

We will denote a digraph $D$ as $D=(V(D),E(D))$, where $V(D)$ is the vertex set and $E(D)$ is the arc set. For clarity, edges and arcs are denoted by using curly braces and parentheses, respectively. For a simple graph $G$, we use $G^*$ to denote symmetric digraph with vertex set $V(G^*)=V(G)$ and arc set $E(G^*)=\bigcup_{\{x,y\}\in E(G)} \{(x,y),(y,x)\}$. Hence, $K_v^*$
and $K_{(x:y)}^*$ respectively denote the complete symmetric digraph of order $v$ and the complete symmetric equipartite digraph with $y$ parts of size $x$. Also, $\vv{C}_n$ will denote the directed cycle of order $n$.

Similarly, a set $\{H_{1}, H_{2}, \ldots, H_{k}\}$ of arc-disjoint subdigraphs of a digraph $D$ is called a decomposition of $D$ if  $\bigcup_{i=1}^{k} E\left(H_{i}\right)=E(D)$. If a symmetric digraph $G^*$ has decomposition which consists of $r_{i}$ factors having directed $n_{i}$ cycles for $i\in [1, t]$, then we say $G^*$ has a $\left\{\vv{C}_{n_{1}}^{r_{1}}, \vv{C}_{n_{2}}^{r_{2}}, \ldots, \vv{C}_{n_{t}}^{r_{t}}\right\}$-factorization.

In the Directed versions of the Oberwolfach and the Hamilton-Waterloo Problems, $K_v^*$ is decomposed into factors of directed cycles. Hence, the seating arrangement is done over $v-1$ nights. If the sizes of directed cycles are $m_{1}, m_{2}, \ldots, m_{t}$ and the number of each directed cycle $m_i$ is $k_i$ for $i\in [1, t]$ where $\sum_{i=1}^{t} k_i m_{i}=v$, the Directed Oberwolfach Problem is denoted by OP$^{*}(m_{1}^{k_1}, \ldots, m_{t}^{k_t})$. Similarly, HWP$^{*}(v; m^{r}, n^{s})$ denotes the uniform directed Hamilton-Waterloo Problem with directed cycle sizes $m$ and $n$. Again, if HWP$^{*}(v; m^{r}, n^{s})$ has a solution, it means that $K_{v}^{*}$ has a $\{\vv{C}_{m}^{r}, \vv{C}_{n}^{s}\}$-factorization for all $r$ and $s$ with $r+s = v-1$.

So far, the Directed Oberwolfach Problem has only partial results, but the Directed Hamilton-Waterloo Problem has not been studied yet up to our knowledge.

As the first result on the Directed Oberwolfach Problem, OP$^{*}(3^k)$ with an exception $v= 6$ is solved by Bermond et al. \cite{Bermond1979}. In \cite{Bennett1990}, Bennett and Zhang solved OP$^{*}(4^k)$ except for $v=12$, and Adams and Bryant solved the remaining case OP$^{*}(4^3)$ (in an unpublished paper ``Resolvable directed cycle systems of all indices for cycle length 3 and 4'').

In \cite{Alspach2003}, Alspach et al. showed that $K_{v}^{*}$ can be decomposed into $\vv{C}_{m}$ cycles with exceptions $(v, m) \neq(4,4),(6,3),(6,6)$ if and only if $m\vert v(v-1)$. They studied the problem in cases where $v$ and $m$ are even or odd, separately.

Burgess and Sajna \cite{Sajna2014} investigated the necessary and sufficient conditions for the Directed Oberwolfach Problem with cycles of length $m$. In case $m$ is even, they obtained complete solution and presented a partial solution for odd cycle size. Also, they conjectured that $K_{2 m}^{*}$ admits a directed $m$-cycle factorization for odd $m$ if and only if $m \geq 5$. In \cite{Sajna2018}, Burgess et al. proved this conjecture for $m \leq 49$. 

The following theorem summarizes the results of Bermond et al. and Burgess and Sajna.

\begin{theorem}\cite{Bermond1979, Sajna2014}\label{OP}
Let $m$ be even, or $m$ and $k$ be odd integers with $m \geq 3$. Then $\mathrm{OP}^{*}(m^k)$ has a solution if and only if $(m, k) \neq (6,1)$.
\end{theorem}

In \cite{Sajna2020}, Shabani and Sajna proved that $K_{v}^{*}$ has a $\{\vv{C}_{2}, \vv{C}_{v-2}\}$-factorization for $v \geq 5$ and obtained the necessary and sufficient conditions for $K_{v}^{*}$ to admit a $\{\vv{C}_{m}, \vv{C}_{v-m}\}$-factorization for $2 \leq m \leq v-2$ and for odd $v$. Also they showed that if $v \geq 5$ and $v \equiv 1,3, \text{or} \, 7 \pmod 8$, then $K_{v}^{*}$ has a $\{\vv{C}_{2}, \vv{C}_{2}, \ldots, \vv{C}_{2}, \vv{C}_{3}\}$-factorization.

In this paper, we follow the lead of the first results on the undirected Hamilton-Waterloo Problem and give solutions to the cases with directed cycle sizes $\{(4,6),(4,8),(4,12),\allowbreak(4,16),(6,12),(8,16),(3,5),(3,15),(5,15)\}$. We first give the necessary conditions for a solution to HWP$^{*}(v;m^{r}, n^{s})$ to exist. Second, we make the observation that for any given solution to HWP$(v; m^{r}, n^{s})$, one can construct a solution to HWP$^{*}(v; m^{2r}, n^{2s})$ for odd $v$. Then, we give two different constructions depending on the parity of the cycle sizes. For even cycle sizes, using our construction in Lemma \ref{mainlemma} and the preliminary Lemmata required in the construction, HWP$^{*}(v; m^{r}, n^{s})$ is solved for $(m, n)\in \{(4,6),(4,8),(4,12),\allowbreak(4,16),(6,12),(8,16)\}$ with $r+s=v-1$. For odd cycle sizes, we give new a construction in Lemma \ref{oddmainlemma} when $v$ is odd. Using this construction and the results required for this construction, we state that HWP$^{*}(v; m^{r}, n^{s})$ has a solution for $(m, n)\in \{(3,5),(3,15),(5,15)\}$ for odd $v$. Constructions given in Lemma \ref{mainlemma} and Lemma \ref{oddmainlemma} are general constructions and they can be used to solve the problem also for the other cycle sizes as long as the necessary small cases can be found.

Let's first start with the necessary conditions and move to the preliminary results then.

\begin{lemma}\label{necessary}
If $\mathrm{HWP}^{*}(v; m^{r}, n^{s})$ has a solution then following statements hold,
\begin{enumerate}
\item if $r>0$, $v \equiv 0 \pmod m$, 
\item if $s>0$, $v \equiv 0 \pmod n$, 
\item $r+s=v-1$.
\end{enumerate}
\end{lemma}

\section{Preliminary Results}
If $G_{1}$ and $G_{2}$ are two edge disjoint graphs with $V(G_1)=V(G_2)$, then we use $G_{1} \oplus G_{2}$ to denote the graph on the same vertex set 
with $E\left(G_{1} \oplus G_{2}\right)=E\left(G_{1}\right) \cup E\left(G_{2}\right)$. We will denote the vertex disjoint union of $\alpha$ copies of $G$ by $\alpha G$. Finally, $\overline{K}_{n}$ denotes the empty graph on $n$ vertices.

Let $G$ and $H$ be graphs, the wreath product of $G$ and $H$, denoted by $G \wr H$, is the graph obtained by replacing each vertex $x$ of $G$ with a copy of $H$, say $H_x$, and replacing each edge $\{x, y\}$ of $G$ with the edges joining every vertex of $H_{x}$ to every vertex of $H_{y}$. 

In case $G$ and $H$ are both digraphs, then the $G \wr  H$ is the digraph obtained by replacing each vertex $x$ of $G$ with a copy of $H$, say $H_x$, and replacing each arc $(x, y)$ of $G$ by an arc pointing from every vertex of $H_{x}$ to every vertex of $H_{y}$. For example, $K_x^* \wr \overline{K}_{y} \cong K_{(y:x)}^*$, $\overline{K}_{x}\wr  K_y^*  \cong xK_y^*$ and $\overline{K}_{x} \wr \overline{K}_{y} \cong \overline{K}_{xy}$.

If $G$ has a $\{H_1, H_2, \dots, H_k\}$-decomposition, then $G\wr  \overline{K}_{n}$ has a $\{H_1\wr  \overline{K}_{n}, H_2\wr  \overline{K}_{n}, \dots, H_k\wr  \overline{K}_{n}\}$-decomposition (see \cite{Alspach1989}). Also, for given three graphs  $G$, $H$, and $J$, $(G\wr H) \wr J=G\wr (H \wr J)$, that is, the wreath product is associative (see \textit{p.} 185 of \cite{ProductGraphs}). Note that, the above properties of the wreath product extend to digraphs.

Let $A$ be a finite additive group and let $S$ be a subset of $A$, where $S$ does not contain the identity of $A$. The Directed Cayley graph $\vv{X}(A ; S)$ on $A$ with connection set $S$ is digraph with $V(\vv{X}(A ; S))=A$ and $E(\vv{X}(A ; S))=\{(x,y):x,y\in A, y-x\in S\}$.

The following observation is useful to reduce the number of cases when $v$ is odd.

\begin{observation}\label{obs}
If $\mathrm{HWP}(v; m^{r}, n^{s})$ has a solution for some $r$ and $s$ and $v$ is odd, then $\mathrm{HWP}^{*}(v; m^{2r},$ $ n^{2s})$ has a solution for the same $r$ and $s$.
\end{observation}

A solution for HWP$^{*}(v; m^{2r}, n^{2s})$ is obtained from a solution of HWP$(v;\allowbreak m^{r}, n^{s})$ by taking two copies of each 2-factor and replacing each edge $\{x, y\}$ with the arcs $(x, y)$ and $(y, x)$ in the two 2-factors. 

Similarly, we get an $H^{*}$-factorization of $G^{*}$ from an $H$-factorization of $G$.

\begin{lemma}\label{lemma1.3}
Let $G$ be a graph and $H$ be a subgraph of $G$. If $G$ has an $H$-factorization then, $G^{*}$ has an $H^{*}$-factorization.
\end{lemma}

The following lemma and theorem will be used in the solutions of even and odd cases of HWP$^{*}(v; m^{r}, n^{s})$, respectively.

\begin{lemma}\cite{Sajna2014} \label{lemma2}
 Let $m \geq 4$ be an even integer and $x$ be a positive integer. Then $K_{(\frac{mx}{2}:2)}^{*}$ has a $\vv{C}_{m}$-factorization.
\end{lemma}

\begin{theorem}\cite{Liu2003} \label{liu}
The complete equipartite graph $K_{(x: y)}$ has a $C_{m}$-fac\-tor\-i\-za\-tion for $m \geq 3$ and $x \geq 2$ if and only if $m\vert xy$, $x(y-1)$ is even, $m$ is even if $y=2$ and $(x, y, m) \neq(2,3,3),(6,3,3),(2,6,3),(6,2,6)$.
\end{theorem}

\section{Even Cycle Sizes}
We will make use of the following lemma in the first main construction of this paper.

\begin{lemma}\label{K_xWrK2}
$K_{x}^{*} \wr \overline{K}_{2}$ has a $K_{2}^{*}$-factorization for every integer $x\geq 2$. 
\end{lemma}
\begin{proof}
Notice that $K_{x}^{*} \wr \overline{K}_{2}\cong K_{2x}^{*} -xK_{2}^{*}$. Using Kotzig's $1$-factorization of $K_{2x}$ and Lemma \ref{lemma1.3}, a decomposition of $K_{x}^{*} \wr \overline{K}_{2}$ into $2x-2$ $K_{2}^{*}$-factors is obtained.
\end{proof}

Here we give the main construction that is used to obtain solutions for the even cycle size cases.

\begin{lemma}\label{mainlemma}
Let $m\geq 4$ and $n\geq 4$ be even and $h=lcm(m,n)$. If $\mathrm{HWP}^{*}(h; m^{r'}, n^{s'})$ has a solution for all nonnegative integers $r'$, $s'$ satisfying $r'+s'=h-1$, then there is a solution to $\mathrm{HWP}^{*}(hx; m^r, n^s)$ for all nonnegative integers $r$, $s$, and $x$ with $r+s=hx-1$.
\end{lemma}

\begin{proof}
We can decompose $K_{hx}^*$ as follows:
\begin{eqnarray}\label{3.1}
 K_{hx}^{*} & \cong & 
 xK_h^{*}\oplus \left(K_{x}^{*} \wr \overline{K}_{h}\right)  
 \label{eq1}
\end{eqnarray}
Since $\overline{K}_{h}\cong \overline{K}_{2}\wr \overline{K}_{\frac{h}{2}}$, $K_{x}^{*} \wr \overline{K}_{h}$  is isomorphic to $(K_{x}^{*} \wr \overline{K}_{2})\wr \overline{K}_{\frac{h}{2}}$ by the associativity of the wreath product.
Thus, by Lemma \ref{K_xWrK2},
$K_{x}^{*} \wr \overline{K}_{h}$ can be decomposed into factors each isomorphic to
$K_{2}^{*} \wr \overline{K}_{\frac{h}{2}}$, and since $K_{2}^{*} \wr \overline{K}_{\frac{h}{2}}\cong K_{(\frac{h}{2}:2)}^{*}$, we have a decomposition of $K_{x}^{*} \wr \overline{K}_{h}$ into $2x-2$ $K_{(\frac{h}{2}:2)}^{*}$-factors. 

Now, let $F_0$ be the $K_{h}^{*}$-factor and $F_1,F_2,\dots,F_{2x-2}$ be the $K_{(\frac{h}{2}:2)}^*$-factors of $K_{hx}^{*}$. Since HWP$^{*}(h; m^{r'},n^{s'})$ is assumed to have 
a solution for all nonnegative integers $r'$ and $s'$, $F_0$ has a $\{\vv{C}_{m}^{r'},\vv{C}_{n}^{s'}\}$-factorization for all nonnegative integers $r'$ and $s'$ where $r'+s'=h-1$. Also, by Lemma \ref{lemma2} $K_{(\frac{h}{2}:2)}^*$ has a $\vv{C}_{m}$- and a $\vv{C}_{n}$-factorization for $m, n\geq 4$, so each $F_j$ has a $\{\vv{C}_{m}^{\frac{h}{2}r_j},\vv{C}_{n}^{\frac{h}{2}s_j}\}$-factorization for $j\in \{1,2, \dots, 2x-2\}$, where $r_j, s_j \in \{0, 1\}$ with $r_j+s_j=1$. Those factorizations give us a $\{\vv{C}_{m}^{r},\vv{C}_{n}^{s}\}$-factorization of $K_{hx}^{*}$ where $r=r'+\frac{h}{2}\sum_{j=1}^{2x-2} r_j$ and $s=s'+\frac{h}{2}\sum_{j=1}^{2x-2} s_j$ with $r+s=r'+s'+\frac{h}{2}\sum_{j=1}^{2x-2} (r_j+ s_j)=h-1+\frac{h}{2}(2x-2)=hx-1$. 

Since any nonnegative integer $0\leq r\leq hx-1$ can be written as $r=r'+\frac{h}{2}a$ for integers $0\leq r'\leq h-1$, $0\leq a\leq 2x-2$ and even $h$, a solution to $\mathrm{HWP}^{*}(hx; m^r, n^s)$ exists for each $r \geq0$ and $s \geq0$ satisfying $r+s=hx-1$.
\end{proof}

\begin{lemma}\label{new2}
For every even integer $m\geq2$, $\vv{C}_{m} \wr \overline{K}_{2}$ has a $\vv{C}_{m}$- or $\vv{C}_{2m}$-factorization.
\end{lemma}
\begin{proof}
Let $m \geq 2$ be an integer. We can represent $\vv{C}_{m} \wr \overline{K}_{2}$ as $\vv{X}\big(\mathbb{Z}_2\mathbb{\times Z}_m ; S_1\big)$, the directed Cayley graph over $\mathbb{Z}_2\mathbb{\times Z}_m$ with the connection set $S_1=\{(0,1),$ $(1,1)\}$. Let $\vv{C}_{(1)}=(v_0,v_1, \dots,\allowbreak v_{m-1})$ be a cycle of $\vv{C}_{m} \wr \overline{K}_{2}$, where $v_i=(0,i)$ for $0\leq i\leq m-1$, and it can be checked that $F_1=\vv{C}_{(1)}\cup (\vv{C}_{(1)}+(1,0))$ is a directed $m$-cycle factor of $\vv{C}_{m} \wr \overline{K}_{2}$. Also, let $\vv{C}_{(2)}=(u_0,u_1, \dots,u_{m-1})$ be a cycle of $\vv{C}_{m} \wr \overline{K}_{2}$, where 
$$ u_{i} =
\begin{cases}
      (0,i) & if \,\   i \,\ is \,\  even\\
      (1,i) & if \,\  i \,\  is \,\  odd
\end{cases}$$
for $0\leq i\leq m-1$. It can be checked that $F_2=\vv{C}_{(2)}\cup (\vv{C}_{(2)}+(1,0))$ is a directed $m$-cycle factor of $\vv{C}_{m} \wr \overline{K}_{2}$. $F_1$ and $F_2$ are arc disjoint directed $m$-cycle factors of $\vv{C}_{m} \wr \overline{K}_{2}$. Thus $\{F_1,F_2\}$ is a $\vv{C}_{m}$-factorization of $\vv{C}_{m} \wr \overline{K}_{2}$.

Let $\vv{C}_{(3)}=(v_0,v_1, \dots,\allowbreak v_{2m-1})$ be a cycle of $\vv{C}_{m} \wr \overline{K}_{2}$, where 
$$ v_i =
\begin{cases}
      (0,i) & if \,\  0\leq i\leq m-1\\
      (1,i) & if \,\  m-2\leq i\leq 2m-1
\end{cases}$$
and it can be checked that $F_3=\vv{C}_{(3)}$ and $F_4=\vv{C}_{(3)}+(1,0)$ are arc disjoint directed $2m$-cycle factor of $\vv{C}_{m} \wr \overline{K}_{2}$. Thus $\{F_3,F_4\}$ is a $\vv{C}_{2m}$-factorization of $\vv{C}_{m} \wr \overline{K}_{2}$.
\end{proof}

For $m\geq 2$, we can represent $(\vv{C}_{m} \wr \overline{K}_{2}) \oplus mK_{2}^{*}$ as the directed Cayley graph over $\mathbb{Z}_2\mathbb{\times Z}_m$ with the connection set $S_2=\{(0,1),(1,0),(1,1)\}$  where $K_{2}^{*}$ consists of edges between $(0,i)$ and $(1,i)$ for $0\leq i \leq m-1$. For brevity, we will denote $(\vv{C}_{m} \wr \overline{K}_{2}) \oplus mK_{2}^{*}$ by $\Gamma_m$.

\begin{lemma}\label{new}
For every integer $m\geq2$, $\Gamma_m$ has a $\{\vv{C}_{m}^{1}, \vv{C}_{2m}^{2}\}$-factorization.
\end{lemma}

\begin{proof}
Let $\vv{C}_{(1)}=(v_0,v_1, \dots,\allowbreak v_{m-1})$ be a cycle of $\Gamma_m$, where $v_i=(0,i)$ for $0\leq i\leq m-1$, and it can be checked that $F_1=\vv{C}_{(1)}\cup (\vv{C}_{(1)}+(1,0))$ is a directed $m$-cycle factor of $\Gamma_m$. Also, let $\vv{C}_{(2)}=(u_0,u_1, \dots,u_{2m-1})$ be a cycle of $\Gamma_m$, where $u_{2i} = (0,i)$, and $u_{2i+1} = (1,i)$ for $0\leq i\leq m-1$. Similarly, it can be checked that $F_2=\vv{C}_{(2)}$ and $F_3=\vv{C}_{(2)}+(1,0)$ are arc disjoint directed $2m$-cycle factors of $\Gamma_m$. Thus $\{F_1,F_2,F_3\}$ is a $\{\vv{C}_{m}^{1}, \vv{C}_{2m}^{2}\}$-factorization of $\Gamma_m$.
\end{proof}

Following Lemmata give the base blocks of our main construction. The cases when $r=0$ and $s=0$ of the Lemmata are obtained by Theorem \ref{OP} and the remaining factorizations for Lemma \ref{lemma4.8} and \ref{lemma3} are given in the Appendix.

\begin{lemma}\label{lemma4.8}
For nonnegative integers $r$ and $s$, $\mathrm{HWP}^{*}(8; 4^{r}, 8^{s})$ has a solution if and only if $r+s=7$.
\end{lemma}

\begin{lemma}\label{lemma3}
For nonnegative integers $r$ and $s$, $\mathrm{HWP}^{*}(12; m^{r}, n^{s})$ has a solution for $(m,n)\in \{(4,6),(4,12)$, $(6,12)\}$ if and only if $r+s=11$.
\end{lemma}

\begin{lemma}\label{lemma6}
For nonnegative integers $r$ and $s$, $\mathrm{HWP}^{*}(16; m^{r}, n^{s})$ has a solution for $(m,n)\in \{(4,16),(8,16)\}$ if and only if $r+s=15$.
\end{lemma}

\begin{proof}
By Theorem \ref{OP}, the cases when $r=0$ and $s=0$ are obtained.\\
\textbf{Case 1 :} $(m,n)=(8,16)$:

We will first analyse when $r$ is odd. We have that $K_{16}^{*} \cong (K_{8}^{*}\wr \overline{K}_{2})\oplus 8K_{2}^{*}$ by \eqref{3.1}, and $K_{8}^{*}$ have a $\vv{C}_{8}$-factorization by Lemma \ref{lemma4.8}. Then, we have a factorization of $K_{16}^{*}$ into six $\vv{C}_{8} \wr \overline{K}_{2}$ and a single $\Gamma_8$ factor.
Also, each $\vv{C}_{8}\wr \overline{K}_{2}$ can be decomposed into two $\vv{C}_{8}$ or two $\vv{C}_{16}$-factors by Lemma \ref{new2}. By Lemma \ref{new}, $\Gamma_8$ has a $\{\vv{C}_{8}^{1}, \vv{C}_{16}^{2}\}$-factorization. Now, let $r_0$ and $s_0$ be nonnegative integers with $r_0+s_0=6$. Decomposing $r_0$ many $\vv{C}_{8}\wr \overline{K}_{2}$'s into $\vv{C}_{8}$-factors and remaining $s_0$ many $\vv{C}_{8}\wr \overline{K}_{2}$'s into $\vv{C}_{16}$-factors, as well as $\Gamma_8$ into a $\{\vv{C}_{8}^{1}, \vv{C}_{16}^{2}\}$-factor gives us a $\{\vv{C}_{8}^{2r_0+1}, \vv{C}_{16}^{2s_0+2}\}$-factorization of $K_{16}^{*}$.

Since any odd integer $r$ can be written as $r=2r_0+1$ for nonnegative integer $r_0$, $\mathrm{HWP}^{*}(16; 8^{r}, 16^{s})$ has a solution for odd $r$ with $r+s=2r_0+1+2s_0+2=2(r_0+s_0)+3=15$.

We list the solutions to the remaining even cases in the Appendix.\\
\textbf{Case 2 :} For $(m,n)=(4,16)$, solutions to all cases are given in the Appendix except for $r=0$ and $s=0$.
\end{proof}

\begin{theorem}
For nonnegative integers $r$ and $s$, $\mathrm{HWP}^{*}(v; m^{r}, n^{s})$ has a solution for $(m,n)\in \{(4,6),(4,8),$
$(4,12),(4,16),(6,12),(8,16)\}$ if and only if $r+s=v-1$ and $lcm(m,n)\vert v$.
\end{theorem}

\begin{proof}
If a solution to HWP$^{*}(v; m^{r}, n^{s})$ exists for $(m,n)\in \{(4,6),(4,8),\allowbreak(4,12),(4,16),(6,12),(8,16)\}$, then by Lemma \ref{necessary} we have $r+s=v-1$, and since $m\vert v$ and $n\vert v$ we have $h=lcm(m,n)\vert v$.

For the sufficiency part, assume $h\vert v$ and $r+s=hx-1=v-1$ where $x$ is a nonnegative integer. 

For $(m,n)=(4,8)$, HWP$^{*}(8; 4^{r_0}, 8^{s_0})$ has a solution for all nonegative $r_0$ and $s_0$ with $r_0+s_0=7$ by Lemma \ref{lemma4.8}. Then, HWP$^{*}(v; 4^{r}, 8^{s})$ has a solution for $r+s=8x-1=v-1$ by Lemma \ref{mainlemma}.

For $(m,n)\in\{(4,6),(4,12),(6,12)\}$, HWP$^{*}(12; m^{r_1}, n^{s_1})$ has a solution for all nonegative $r_1$ and $s_1$ with $r_1+s_1=11$ by Lemma \ref{lemma3}. Then, HWP$^{*}(v; m^{r}, n^{s})$ has a solution by Lemma \ref{mainlemma} for $(m,n)\in \{(4,6),(4,12)$, $(6,12)\}$ with $r+s=12x-1=v-1$.

For $(m,n)\in \{(4,16),(8,16)\}$, HWP$^{*}(16; m^{r_2}, n^{s_2})$ has a solution for all nonegative $r_2$ and $s_2$ with $r_2+s_2=15$ by Lemma \ref{lemma6}. Then, by Lemma \ref{mainlemma}, HWP$^{*}(v; m^{r}, n^{s})$ has a solution for $(m,n)\in \{(4,16),(8,16)\}$ with $r+s=16x-1=v-1$.
\end{proof}

\section{Odd Cycle Sizes}

Here we first give the following main construction, and using this construction we prove that HWP$^{*}(v; m^{r}, n^{s})$ has a solution for $(m,n)\in \{(3,5),(3,15),\allowbreak(5,15)\}$ with $r+s=v-1$, where $v$ is odd.

\begin{lemma}\label{oddmainlemma}
Let $m\geq3$ and $n\geq3$ be odd, $h=lcm(m,n)$ and $3\vert h$. If $\mathrm{HWP}^{*}(h; m^{r'}, n^{s'})$ has a solution for all $r'$, $s'$ satisfying $r'+s'=h-1$, then there is a solution to $\mathrm{HWP}^{*}(hx; m^{r}, n^{s})$ for all nonnegative $r$, $s$ and odd $x$ satisfying $r+s=hx-1$.
\end{lemma}

\begin{proof}
By \eqref{3.1}, we have a decomposition of $K_{hx}^{*}$ into a $K_h^{*}$ and a $\left(K_{x}^{*} \wr \overline{K}_{h}\right)$-factor. Since $\overline{K}_{h}\cong \overline{K}_{3}\wr \overline{K}_{\frac{h}{3}}$, we have $K_{x}^{*} \wr \overline{K}_{h} \cong (K_{x}^{*} \wr \overline{K}_{3})\wr \overline{K}_{\frac{h}{3}}$. 

It is clear that $K_{x}^{*} \wr \overline{K}_{3}$ is isomorphic to $K_{3x}^{*} -xK_{3}^{*}$. Since Kirkman triple system of order $3x$ exists, we have a $C_3$-factorization of $K_{3x}$. Then, a $C_{3}^{*}\cong K_{3}^{*}$-factorization of $K_{3x}^{*} -xK_{3}^{*}$ is obtained by Lemma \ref{lemma1.3}. So, $K_{x}^{*} \wr \overline{K}_{3}$ has a decomposition into $\frac{3x-3}{2}$ $K_{3}^{*}$-factors. In $K_{x}^{*} \wr \overline{K}_{h}$, these $K_{3}^{*}$-factors form $K_{(\frac{h}{3}:3)}^{*}$-factors since $K_{3}^{*}\wr \overline{K}_{\frac{h}{3}} \cong K_{(\frac{h}{3}:3)}^{*}$.

Let $F_0$ be the $K_{h}^*$-factor and $F_1,F_2,\dots,F_{\frac{3x-3}{2}}$ be the $K_{(\frac{h}{3}:3)}^*$-factors of $K_{hx}^{*}$. Since HWP$^{*}(h; m^{r'}, n^{s'})$ is assumed to have a solution for all nonnegative integers $r'$ and $s'$ where $r'+s'=h-1$, $F_0$ has a $\{\vv{C}_{m}^{r'},\vv{C}_{n}^{s'}\}$-factorization for all nonnegative integers $r'$ and $s'$ with $r'+s'=h-1$. Also, $K_{(\frac{h}{3}:3)}^*$ has a $\vv{C}_{m}$-factorization and a $\vv{C}_{n}$-factorization by Lemma \ref{lemma1.3} and Theorem \ref{liu}, so each $F_j$ has a $\{\vv{C}_{m}^{\frac{2h}{3}r_j},\vv{C}_{n}^{\frac{2h}{3}s_j}\}$-factorization for $j\in \{1,2, \dots,\frac{3x-3}{2}\}$, where $r_j, s_j \in \{0, 1 \}$ with $r_j+s_j=1$. These factorizations give us a $\{\vv{C}_{m}^{r},\vv{C}_{n}^{s}\}$-factorization of $K_{hx}^{*}$ where $r=r'+\sum_{i=0}^{\frac{3x-3}{2}} \frac{2h}{3}r_i$ and $s=s'+\sum_{i=0}^{\frac{3x-3}{2}} \frac{2h}{3} s_i$ with $r+s=r'+s'+\sum_{i=0}^{\frac{3x-3}{2}} \frac{2h}{3}(r_i+s_i)=h-1+hx-h=hx-1$.

Since any nonnegative integer $0\leq r\leq hx-1$ can be written as $r=r'+\frac{2h}{3}a$ for integers $0\leq r'\leq h-1$ and $0\leq a\leq \frac{3x-3}{2}$, a solution to $\mathrm{HWP}^{*}(hx; m^r, n^s)$ exists for each $r \geq 0$ and  $s \geq 0$ satisfying $r+s=hx-1$.
\end{proof}

\begin{lemma}\label{lemma3.2}
For nonnegative integers $r$ and $s$, $\mathrm{HWP}^{*}(15; m^{r}, n^{s})$ has a solution for $(m,n)\in \{(3,5),(3,15)$, $(5,15)\}$ if and only if $r+s=14$ except possibly for $r\in \{11,12,13\}$ when $(m, n)=(3, 5)$ and for $r=13$ when $(m, n)=(3, 15)$.
\end{lemma}

\begin{proof}
The cases when $r= 0$ and $s=0$ can be obtained by Theorem \ref{OP}.

In \cite{Adams2002}, a solution to HWP$(15; m^{r_0}, n^{s_0})$ for $(m,n)\in \{(3,5),(3,15),(5,15)\}$ with the exception $(m,n,r_0,s_0)=(3,5,6,1)$ is given by Theorem 4.1. Thus, by Observation \ref{obs}, we have a solution to HWP$^{*}(15; m^{r}, n^{s})$ for $(m,n)\in \{(3,5),(3,15)$, $(5,15)\}$ with $r$ and $s$ are positive even integers except possibly for $(m,n,r,s)=(3,5,12,2)$. We list the solutions for the odd cases in the Appendix.
\end{proof}

Although there are missing cases in here, we can still find a solution of HWP$^{*}(v; m^{r}, n^{s})$ for $(m,n)\in \{(3,5),(3,15),(5,15)\}$ with $r+s=v-1$ and for odd $v>15$, using the main construction.

\begin{theorem}\label{thm4.3}
For all nonnegative integers $r$, $s$ and odd $v>15$, $\mathrm{HWP}^{*}(v;\allowbreak m^{r}, n^{s})$ has a solution for $(m,n)\in \{(3,5),(3,15),(5,15)\}$ if and only if $r+s=v-1$ and $15\vert v$.
\end{theorem}

\begin{proof}
If a solution to HWP$^{*}(v; m^{r}, n^{s})$ exists for $(m,n)\in \{(3,5),(3,15),\allowbreak(5,15)\}$, then by Lemma \ref{necessary}, $r+s=v-1$ and $lcm(m,n)=15\vert v$.

For the sufficiency part, assume $v=15x$ and $r+s=15x-1$ where $x>1$ is an odd integer. 

Using the main construction in Lemma \ref{oddmainlemma} for $v=15x$, we have a decomposition of $K_{v}^{*}$ into a $K_{15}^{*}$ and $\frac{3x-3}{2}$ $K_{(5:3)}^{*}$ factors.
$K_{15}^{*}$ has a $\{\vv{C}_{m}^{r'},\vv{C}_{n}^{s'}\}$-factorization for $(m,n)\in \{(3,5),(3,15),(5,15)\}$ with $r'+s'=14$ except possibly for $r'\in \{11,12,13\}$ when $(m, n)=(3, 5)$ and for $r'=13$ when $(m, n)=(3, 15)$ by Lemma \ref{lemma3.2}. By Lemma \ref{lemma1.3} and Theorem \ref{liu}, each $K_{(5:3)}^{*}$ has a $\{\vv{C}_{m}^{10r_j},\vv{C}_{n}^{10s_j}\}$-factorization for $j\in \{1,2, \dots,\frac{3x-3}{2}\}$, where $r_j, s_j \in \{0, 1 \}$ with $r_j+s_j=1$. Let $r=r'+10a$ and $s=s'+10b$ for nonnegative $a$ and $b$ with $a+b=\frac{3x-3}{2}$, then we have $r+s=r'+s'+10(a+b)=14+5(3x-3)=15x-1=v-1$ with $0\leq r,s\leq 15x-1$. 

We obtain the requested integer $ r\in [0, 15x-1]$ from the sum of $r'$ and $10a$ for integers $0\leq r'\leq 14$ and $0\leq a\leq \frac{3x-3}{2}$. Therefore, $\mathrm{HWP}^{*}(v; m^r, n^s)$ has a solution with $r+s=15x-1=v-1$.
\end{proof}

According to our best knowledge, our results are the first findings for the directed version of the Hamilton-Waterloo Problem. We have first examined the cases $(m, n)\in \{(4,6),(4,8),(4,16),(8,16),(3,5),(3,15),(5,15)\}$, as done in the first paper on the undirected Hamilton-Waterloo Problem by Adams et al. \cite{Adams2002}. We have also solved the problem for the cases $(m, n)\in \{(4, 12), (6, 12)\}$. In addition to studying odd cycle cases $\{(3,5),(3,15),\allowbreak(5,15)\}$, we have also observed that if $\mathrm{HWP}(v; m^{r}, n^{s})$ has a solution for odd $v$, then $\mathrm{HWP}^{*}(v; m^{2r},n^{2s})$ has a solution for the same $r$ and $s$ as well. Since there is no $2$-factorizations of $K_v$ for even $v$, we cannot arrive the similar observation when $v$ is even and $m, n > 2$. Our constructions given in Lemma \ref{mainlemma} and Lemma \ref{oddmainlemma} can also be used to solve the problem for the other cycle sizes as long as the necessary small cases can be found.

Now we can combine our results in the following main theorem.

\begin{theorem}
For nonnegative integers $r$ and $s$, $\mathrm{HWP}^{*}(v; m^{r}, n^{s})$ has a solution for
\begin{enumerate}
\item $(m,n)\in \{(4,6),(4,8),(4,12),(4,16),(6,12),(8,16)\}$ when $v$ is even,
\item $(m,n)\in \{(3,5),(3,15),(5,15)\}$ when $v$ is odd
\end{enumerate}
if and only if $r+s=v-1$ and $lcm(m,n)\vert v$ except possibly for $r\in \{11,12,13\}$ when $(v, m, n)=(15, 3, 5)$ and for $r=13$ when $(v, m, n)=(15, 3, 15)$.
\end{theorem}

\scriptsize
\section{Appendix}\label{secA1}

Let $V(K_{8}^*)= \mathbb{Z}_{8}$, $V(K_{12}^*)= \mathbb{Z}_{12}$, $V(K_{16}^*)= \mathbb{Z}_{16}$ and $V(K_{15}^*)= \mathbb{Z}_{15}$,

\begin{enumerate}[leftmargin=*,align=left]

\item HWP$^*(8;4^1,8^{6})${,}\\
$[(0{,}3{,}2{,}1){,}(4{,}7{,}6{,}5)],[(0{,}3{,}2{,}4{,}6{,}5{,}7{,}1)],[(0{,}2{,}1{,}3{,}5{,}4{,}7{,}6)],[(0{,}7{,}5{,}3{,}6{,}1{,}4{,}2)],[(0{,}6{,}4{,}3{,}1{,}7{,}2{,}5)],$\\
$[(0{,}5{,}1{,}6{,}2{,}7{,}3{,}4)],[(0{,}4{,}1{,}5{,}2{,}6{,}3{,}7)]$
\item HWP$^*(8;4^2,8^{5})${,}\\
$[(0{,}2{,}7{,}6){,}(1{,}4{,}5{,}3)],[(0{,}5{,}6{,}7){,}(1{,}3{,}2{,}4)],[(0{,}1{,}7{,}2{,}6{,}3{,}5{,}4)],[(0{,}3{,}7{,}4{,}6{,}5{,}1{,}2)]{,}[(0{,}4{,}2{,}5{,}7{,}3{,}6{,}1)],$\\
$[(0{,}6{,}2{,}3{,}4{,}7{,}1{,}5)],[(0{,}7{,}5{,}2{,}1{,}6{,}4{,}3)]$
\item HWP$^*(8;4^3,8^{4})${,}\\
$[(0{,}1{,}2{,}3){,}(4{,}5{,}6{,}7)],[(0{,}2{,}4{,}6){,}(1{,}3{,}5{,}7)],[(0{,}3{,}2{,}1){,}(4{,}7{,}6{,}5)],[(0{,}4{,}1{,}5{,}2{,}6{,}3{,}7)],$\\
$[(0{,}4{,}2{,}5{,}7{,}3{,}6{,}1)],[(0{,}6{,}4{,}3{,}1{,}7{,}2{,}5)],[(0{,}7{,}5{,}3{,}6{,}1{,}4{,}2)]$
\item HWP$^*(8;4^4,8^{3})${,}\\
$[(0{,}3{,}7{,}4){,}(1{,}2{,}6{,}5)],[(0{,}1{,}7{,}2){,}(3{,}5{,}4{,}6)],[(0{,}2{,}7{,}6){,}(1{,}4{,}5{,}3)],[(0{,}5{,}6{,}7){,}(1{,}3{,}2{,}4)],$\\
$[(0{,}4{,}2{,}5{,}7{,}3{,}6{,}1)],[(0{,}6{,}2{,}3{,}4{,}7{,}1{,}5)],[(0{,}7{,}5{,}2{,}1{,}6{,}4{,}3)]$
\item HWP$^*(8;4^5,8^{2})${,}\\
$[(0{,}1{,}2{,}3){,}(4{,}5{,}6{,}7)],[(0{,}2{,}4{,}6){,}(1{,}3{,}5{,}7)],[(0{,}3{,}2{,}1){,}(4{,}7{,}6{,}5)],[(0{,}6{,}4{,}2){,}(1{,}7{,}5{,}3)]{,}$\\
$[(0{,}7{,}2{,}5){,}(4{,}3{,}6{,}1)],[(0{,}4{,}1{,}5{,}2{,}6{,}3{,}7)],[(0{,}5{,}1{,}6{,}2{,}7{,}3{,}4)]$
\item HWP$^*(8;4^6,8^{1})${,}\\
$[(0{,}2{,}1{,}3){,}(4{,}7{,}6{,}5)],[(0{,}3{,}7{,}4){,}(1{,}5{,}2{,}6)],[(0{,}4{,}1{,}6){,}(2{,}7{,}5{,}3)],[(0{,}5{,}7{,}1){,}(2{,}4{,}3{,}6)]{,}$\\
$[(0{,}6{,}4{,}2){,}(1{,}7{,}3{,}5)],[(0{,}7{,}2{,}5){,}(1{,}4{,}6{,}3)],[(0{,}1{,}2{,}3{,}4{,}5{,}6{,}7)]$
\item HWP$^*(12;4^1,6^{10})${,}\\
$[(0{,}2{,}1{,}3{,}4{,}6){,}(5{,}7{,}8{,}10{,}9{,}11)],[(0{,}3{,}1{,}4{,}2{,}5){,}(6{,}8{,}11{,}9{,}7{,}10)],[(0{,}4{,}1{,}5{,}2{,}7){,}(3{,}9{,}6{,}11{,}10{,}8)],$\\
$[(0{,}6{,}1{,}7{,}2{,}9){,}(3{,}8{,}4{,}10{,}5{,}11)],[(0{,}7{,}1{,}8{,}2{,}10){,}(3{,}11{,}6{,}5{,}9{,}4)],[(0{,}8{,}6{,}3{,}2{,}4){,}(1{,}9{,}5{,}10{,}7{,}11)],$\\
$[(0{,}9{,}1{,}10{,}2{,}11){,}(3{,}7{,}6{,}4{,}8{,}5)],[(0{,}10{,}3{,}6{,}9{,}2),(1{,}11{,}4{,}7{,}5{,}8)],[(0{,}11{,}2{,}6{,}10{,}1){,}(3{,}5{,}4{,}9{,}8{,}7)],$\\
$[(0{,}1{,}2{,}3){,}(4{,}5{,}6{,}7){,}(8{,}9{,}10{,}11)]$
\item HWP$^*(12;4^2,6^{9})${,}\\
$[(0{,}3{,}1{,}5{,}2{,}6){,}(4{,}7{,}10{,}8{,}11{,}9)],[(0{,}4{,}1{,}3{,}2{,}5){,}(6{,}8{,}10{,}9{,}7{,}11)],[0{,}5{,}1{,}6{,}2{,}8){,}(3{,}10{,}4{,}11{,}7{,}9)],$\\
$[(0{,}6{,}1{,}7{,}2{,}9){,}(3{,}8{,}4{,}10{,}5{,}11)],[(0{,}7{,}1{,}9{,}6{,}10){,}(2{,}4{,}3{,}11{,}5{,}8)],[(0{,}8{,}1{,}11{,}2{,}7){,}(3{,}6{,}4{,}9{,}5{,}10)],$\\
$[(0{,}9{,}1{,}10{,}2{,}11){,}(3{,}4{,}8{,}7{,}6{,}5)],[(0{,}10{,}7{,}3{,}9{,}2){,}(1{,}8{,}5{,}4{,}6{,}11)],[(0{,}11{,}4{,}2{,}10{,}1){,}(3{,}7{,}5{,}9{,}8{,}6)],$\\
$[(0{,}1{,}2{,}3){,}(4{,}5{,}6{,}7){,}(8{,}9{,}10{,}11)],[(0{,}2{,}1{,}4){,}(3{,}5{,}7{,}8){,}(6{,}9{,}11{,}10)]$
\item HWP$^*(12;4^3,6^{8})${,}\\
$[(0{,}3{,}1{,}6{,}2{,}7){,}(4{,}8{,}10{,}5{,}11{,}9)],[(0{,}5{,}1{,}3{,}2{,}8){,}(4{,}6{,}10{,}9{,}7{,}11)],[(0{,}11{,}1{,}10{,}8{,}2){,}(3{,}4{,}7{,}6{,}5{,}9)],$\\
$[(0{,}6{,}1{,}7{,}2{,}9){,}(3{,}8{,}4{,}11{,}5{,}10)],[(0{,}7{,}1{,}11{,}2{,}10){,}(3{,}9{,}6{,}8{,}5{,}4)],[(0{,}8{,}7{,}3{,}11{,}6){,}(1{,}9{,}5{,}2{,}4{,}10)],$\\
$[(0{,}9{,}1{,}8{,}6{,}11){,}(2{,}5{,}3{,}7{,}10{,}4)],[(0{,}10{,}7{,}5{,}8{,}1){,}(2{,}11{,}3{,}6{,}4{,}9)],[(0{,}1{,}2{,}3){,}(4{,}5{,}6{,}7){,}(8{,}9{,}10{,}11)],$\\
$[(0{,}2{,}1{,}4){,}(3{,}5{,}7{,}8){,}(6{,}9{,}11{,}10)],[(0{,}4{,}1{,}5){,}(2{,}6{,}3{,}10){,}(7{,}9{,}8{,}11)]$
\item HWP$^*(12;4^4,6^{7})${,}\\
$[(0{,}3{,}1{,}7{,}2{,}8){,}(4{,}6{,}10{,}5{,}11{,}9)],[(0{,}6{,}1{,}3{,}2{,}9){,}(4{,}11{,}5{,}8{,}7{,}10)],[(0{,}7{,}1{,}8{,}4{,}10){,}(2{,}11{,}6{,}5{,}3{,}9)],$\\
$[(0{,}8{,}6{,}4{,}3{,}11){,}(1{,}9{,}7{,}5{,}2{,}10)],[(0{,}9{,}5{,}10{,}8{,}1){,}(2{,}7{,}6{,}11{,}3{,}4)],[(0{,}10{,}7{,}3{,}6{,}2){,}(1{,}11{,}4{,}8{,}5{,}9)],$\\
$[(0{,}11{,}1{,}10{,}3{,}7){,}(2{,}5{,}4{,}9{,}6{,}8)],[(0{,}1{,}2{,}3){,}(4{,}5{,}6{,}7){,}(8{,}9{,}10{,}11)],[(0{,}2{,}1{,}4){,}(3{,}5{,}7{,}8){,}(6{,}9{,}11{,}10)],$\\
$[(0{,}4{,}1{,}5){,}(2{,}6{,}3{,}10){,}(7{,}9{,}8{,}11)],[(0{,}5{,}1{,}6){,}(2{,}4{,}7{,}11){,}(3{,}8{,}10{,}9)]$
\item HWP$^*(12;4^5,6^{6})${,}\\
$[(0{,}3{,}1{,}8{,}4{,}2){,}(5{,}11{,}9{,}6{,}10{,}7)],[(0{,}7{,}1{,}9{,}2{,}8){,}(3{,}6{,}11{,}5{,}10{,}4)],[(0{,}8{,}5{,}9{,}1{,}10){,}(2{,}7{,}6{,}4{,}11{,}3)],$\\
$[(0{,}9{,}7{,}10{,}1{,}11){,}(2{,}5{,}3{,}4{,}8{,}6)],[(0{,}10{,}5{,}2{,}11{,}1){,}(3{,}9{,}4{,}6{,}8{,}7)],[(0{,}11{,}6{,}5{,}4{,}9){,}(1{,}3{,}7{,}2{,}10{,}8)],$\\
$[(0{,}1{,}2{,}3){,}(4{,}5{,}6{,}7){,}(8{,}9{,}10{,}11)],[(0{,}2{,}1{,}4){,}(3{,}5{,}7{,}8){,}(6{,}9{,}11{,}10)],[(0{,}4{,}1{,}5){,}(2{,}6{,}3{,}10){,}(7{,}9{,}8{,}11)],$\\
$[(0{,}5{,}1{,}6){,}(2{,}4{,}7{,}11){,}(3{,}8{,}10{,}9)],[(0{,}6{,}1{,}7){,}(2{,}9{,}5{,}8){,}(3{,}11{,}4{,}10)]$
\item HWP$^*(12;4^6,6^{5})${,}\\
$[(0{,}3{,}1{,}9{,}6{,}8){,}(2{,}7{,}10{,}4{,}11{,}5)],[(0{,}8{,}7{,}6{,}11{,}9){,}(1{,}3{,}4{,}2{,}5{,}10)],[(0{,}9{,}7{,}5{,}11{,}1){,}(2{,}10{,}8{,}6{,}4{,}3)],$\\
$[(0{,}10{,}7{,}3{,}9{,}2){,}(1{,}11{,}6{,}5{,}4{,}8)],[(0{,}11{,}3{,}7{,}1{,}10){,}(2{,}8{,}5{,}9{,}4{,}6)],[(0{,}1{,}2{,}3){,}(4{,}5{,}6{,}7){,}(8{,}9{,}10{,}11)],$\\
$[(0{,}2{,}1{,}4){,}(3{,}5{,}7{,}8){,}(6{,}9{,}11{,}10)],[(0{,}4{,}1{,}5){,}(2{,}6{,}3{,}10){,}(7{,}9{,}8{,}11)],[(0{,}5{,}1{,}6){,}(2{,}4{,}7{,}11){,}(3{,}8{,}10{,}9)],$\\
$[(0{,}6{,}1{,}7){,}(2{,}9{,}5{,}8){,}(3{,}11{,}4{,}10)],[(0{,}7{,}2{,}11){,}(1{,}8{,}4{,}9){,}(3{,}6{,}10{,}5)]$
\item HWP$^*(12;4^7,6^{4})${,}\\
$[(0{,}3{,}9{,}2{,}7{,}1){,}(4{,}11{,}5{,}10{,}8{,}6)],[(0{,}9{,}4{,}3{,}2{,}8){,}(1{,}10{,}7{,}6{,}5{,}11)],[(0{,}10{,}4{,}2{,}5{,}9){,}(1{,}11{,}6{,}8{,}7{,}3)],$\\
$[(0{,}11{,}9{,}6{,}2{,}10){,}(1{,}3{,}7{,}5{,}4{,}8)],[(0{,}1{,}2{,}3){,}(4{,}5{,}6{,}7){,}(8{,}9{,}10{,}11)],[(0{,}2{,}1{,}4){,}(3{,}5{,}7{,}8){,}(6{,}9{,}11{,}10)],$\\
$[(0{,}4{,}1{,}5){,}(2{,}6{,}3{,}10){,}(7{,}9{,}8{,}11)],[(0{,}5{,}1{,}6){,}(2{,}4{,}7{,}11){,}(3{,}8{,}10{,}9)],[(0{,}6{,}1{,}7){,}(2{,}9{,}5{,}8){,}(3{,}11{,}4{,}10)],$\\
$[(0{,}7{,}2{,}11){,}(1{,}8{,}4{,}9){,}(3{,}6{,}10{,}5)],[0{,}8{,}5{,}2){,}(1{,}9{,}7{,}10){,}(3{,}4{,}6{,}11)]$
\item HWP$^*(12;4^8{,}6^{3})${,}\\
$[(0{,}3{,}9{,}2{,}7{,}1){,}(4{,}11{,}5{,}10{,}8{,}6)],[(0{,}9{,}4{,}3{,}2{,}8){,}(1{,}10{,}7{,}6{,}5{,}11)],[(0{,}10{,}4{,}2{,}5{,}9){,}(1{,}11{,}6{,}8{,}7{,}3)],$\\
$[(0{,}1{,}2{,}3){,}(4{,}5{,}6{,}7){,}(8{,}9{,}10{,}11)],[(0{,}2{,}1{,}4){,}(3{,}5{,}7{,}8){,}(6{,}11{,}10{,}9)],[(0{,}4{,}1{,}6){,}(2{,}11{,}9{,}7){,}(3{,}8{,}10{,}5)],$\\
$[(0{,}5{,}1{,}7){,}(2{,}6{,}9{,}8){,}(3{,}11{,}4{,}10)],[(0{,}6{,}2{,}10){,}(1{,}9{,}5{,}8){,}(3{,}4{,}7{,}11)],[(0{,}7{,}9{,}11){,}(1{,}3{,}6{,}10){,}(2{,}4{,}8{,}5)],$\\
$[(0{,}8{,}11{,}2){,}(1{,}5{,}4{,}9){,}(3{,}7{,}10{,}6)],[(0{,}11{,}7{,}5){,}(1{,}8{,}4{,}6){,}(2{,}9{,}3{,}10)]$
\item HWP$^*(12;4^{9},6^{2})${,}\\
$[(0{,}3{,}9{,}2{,}7{,}1){,}(4{,}11{,}5{,}10{,}8{,}6)],[(0{,}9{,}4{,}3{,}2{,}8){,}(1{,}10{,}7{,}6{,}5{,}11)],[(0{,}1{,}2{,}3){,}(4{,}5{,}6{,}7){,}(8{,}9{,}10{,}11)],$\\
$[(0{,}2{,}1{,}4){,}(3{,}5{,}7{,}8){,}(6{,}9{,}11{,}10)],[(0{,}4{,}1{,}5){,}(2{,}6{,}3{,}10){,}(7{,}9{,}8{,}11)],[(0{,}5{,}1{,}7){,}(2{,}10{,}4{,}8){,}(3{,}6{,}11{,}9)],$\\
$[(0{,}6{,}2{,}11){,}(1{,}9{,}7{,}3){,}(4{,}10{,}5{,}8)],[(0{,}7{,}11{,}2){,}(1{,}8{,}5{,}9){,}(3{,}4{,}6{,}10)],[(0{,}8{,}7{,}10){,}(1{,}3{,}11{,}6){,}(2{,}9{,}5{,}4)],$\\
$[(0{,}10{,}9{,}6){,}(1{,}11{,}3{,}8){,}(2{,}4{,}7{,}5)],[(0{,}11{,}4{,}9){,}(1{,}6{,}8{,}10){,}(2{,}5{,}3{,}7)]$
\item HWP$^*(12;4^{10},6^{1})${,}\\
$[(0{,}3{,}9{,}2{,}7{,}1){,}(4{,}11{,}5{,}10{,}8{,}6)],[(0{,}1{,}2{,}3){,}(4{,}5{,}6{,}7){,}(8{,}9{,}10{,}11)],[(0{,}2{,}1{,}4){,}(3{,}5{,}7{,}6){,}(8{,}11{,}10{,}9)],$\\
$[(0{,}4{,}1{,}5){,}(2{,}6{,}8{,}10){,}(3{,}7{,}9{,}11)],[(0{,}5{,}1{,}6){,}(2{,}8{,}3{,}11){,}(4{,}9{,}7{,}10)],[(0{,}6{,}5{,}2){,}(1{,}9{,}3{,}10){,}(4{,}8{,}7{,}11)],$\\
$[(0{,}7{,}2{,}10){,}(1{,}8{,}5{,}3){,}(4{,}6{,}11{,}9)],[(0{,}8{,}1{,}11){,}(2{,}5{,}9{,}6){,}(3{,}4{,}10{,}7)],[(0{,}9{,}1{,}7){,}(2{,}4{,}3{,}8){,}(5{,}11{,}6{,}10)],$\\
$[(0{,}10{,}6{,}9){,}(1{,}3{,}2{,}11){,}(4{,}7{,}5{,}8)],[(0{,}11{,}7{,}8){,}(1{,}10{,}3{,}6){,}(2{,}9{,}5{,}4)]$
\item HWP$^*(12;4^1,12^{10})${,}\\
$[(0{,}1{,}2{,}3){,}(4{,}5{,}6{,}7){,}(8{,}9{,}10{,}11)],[(0{,}2{,}1{,}3{,}4{,}6{,}5{,}7{,}8{,}10{,}9{,}11)],[(0{,}3{,}1{,}4{,}2{,}5{,}8{,}6{,}9{,}7{,}11{,}10)]{,}$\\
$[(0{,}4{,}1{,}5{,}2{,}6{,}3{,}7{,}10{,}8{,}11{,}9)],[(0{,}5{,}1{,}6{,}2{,}4{,}10{,}3{,}11{,}7{,}9{,}8)],[(0{,}6{,}1{,}8{,}2{,}9{,}3{,}10{,}5{,}11{,}4{,}7)]{,}$\\
$[(0{,}7{,}1{,}9{,}2{,}8{,}3{,}5{,}10{,}4{,}11{,}6)],[(0{,}8{,}1{,}7{,}3{,}9{,}6{,}10{,}2{,}11{,}5{,}4)],[(0{,}9{,}4{,}8{,}5{,}3{,}6{,}11{,}1{,}10{,}7{,}2)]{,}$\\
$[(0{,}10{,}1{,}11{,}3{,}2{,}7{,}6{,}8{,}4{,}9{,}5)],[(0{,}11{,}2{,}10{,}6{,}4{,}3{,}8{,}7{,}5{,}9{,}1)]$
\item HWP$^*(12;4^2,12^{9})${,}\\
$[(0{,}1{,}2{,}3){,}(4{,}5{,}6{,}7){,}(8{,}9{,}10{,}11)],[(0{,}2{,}4{,}6){,}(8{,}10{,}1{,}3){,}(5{,}7{,}9{,}11)],[(0{,}3{,}1{,}4{,}2{,}5{,}8{,}6{,}9{,}7{,}11{,}10)],$\\
$[(0{,}4{,}1{,}5{,}2{,}6{,}3{,}7{,}10{,}8{,}11{,}9)],[(0{,}5{,}1{,}6{,}2{,}7{,}3{,}10{,}9{,}8{,}4{,}11)],[(0{,}6{,}1{,}7{,}2{,}9{,}5{,}10{,}3{,}11{,}4{,}8)],$\\
$[(0{,}7{,}1{,}8{,}2{,}10{,}4{,}9{,}6{,}11{,}3{,}5)],[(0{,}8{,}1{,}9{,}2{,}11{,}6{,}5{,}3{,}4{,}10{,}7)],[(0{,}9{,}1{,}11{,}7{,}8{,}5{,}4{,}3{,}6{,}10{,}2)],$\\
$[(0{,}10{,}6{,}8{,}3{,}9{,}4{,}7{,}5{,}11{,}2{,}1)],[(0{,}11{,}1{,}10{,}5{,}9{,}3{,}2{,}8{,}7{,}6{,}4)]$
\item HWP$^*(12;4^3,12^{8})${,}\\
$[(0{,}1{,}2{,}3){,}(4{,}5{,}6{,}7){,}(8{,}9{,}10{,}11)],[(0{,}2{,}4{,}6){,}(8{,}10{,}1{,}3){,}(5{,}7{,}9{,}11)],[(0{,}3{,}6{,}1){,}(7{,}10{,}2{,}8){,}(4{,}9{,}5{,}11)],$\\
$[(0{,}4{,}1{,}5{,}2{,}6{,}3{,}7{,}8{,}11{,}10{,}9)],[(0{,}5{,}1{,}4{,}2{,}7{,}3{,}10{,}6{,}11{,}9{,}8)],[(0{,}6{,}2{,}1{,}7{,}5{,}9{,}4{,}10{,}8{,}3{,}11)],$\\
$[(0{,}7{,}1{,}6{,}5{,}8{,}4{,}11{,}3{,}9{,}2{,}10)],[(0{,}8{,}1{,}9{,}3{,}2{,}5{,}10{,}7{,}11{,}6{,}4)],[(0{,}9{,}1{,}8{,}2{,}11{,}7{,}6{,}10{,}4{,}3{,}5)],$\\
$[(0{,}10{,}3{,}1{,}11{,}2{,}9{,}6{,}8{,}5{,}4{,}7)],[(0{,}11{,}1{,}10{,}5{,}3{,}4{,}8{,}6{,}9{,}7{,}2)]$
\item HWP$^*(12;4^4,12^{7})${,}\\
$[(0{,}1{,}2{,}3){,}(4{,}5{,}6{,}7){,}(8{,}9{,}10{,}11)],[(0{,}2{,}4{,}6){,}(8{,}10{,}1{,}3){,}(5{,}7{,}9{,}11)],[(0{,}3{,}6{,}1){,}(7{,}10{,}2{,}8){,}(4{,}9{,}5{,}11)],$\\
$[(0{,}4{,}1{,}10){,}(2{,}5{,}3{,}9){,}(11{,}7{,}6{,}8)],[(0{,}5{,}1{,}4{,}2{,}6{,}3{,}7{,}11{,}10{,}9{,}8)],[(0{,}6{,}2{,}1{,}5{,}4{,}10{,}7{,}8{,}3{,}11{,}9)],$\\
$[(0{,}7{,}1{,}6{,}4{,}11{,}2{,}9{,}3{,}10{,}8{,}5)],[(0{,}8{,}1{,}7{,}2{,}10{,}4{,}3{,}5{,}9{,}6{,}11)],[(0{,}9{,}1{,}11{,}6{,}10{,}5{,}8{,}4{,}7{,}3{,}2)],$\\
$[(0{,}10{,}3{,}4{,}8{,}6{,}5{,}2{,}11{,}1{,}9{,}7)],[(0{,}11{,}3{,}1{,}8{,}2{,}7{,}5{,}10{,}6{,}9{,}4)]$
\item HWP$^*(12;4^5,12^{6})${,}\\
$[(0{,}1{,}2{,}3){,}(4{,}5{,}6{,}7){,}(8{,}9{,}10{,}11)],[(0{,}2{,}4{,}6){,}(8{,}10{,}1{,}3){,}(5{,}7{,}9{,}11)],[(0{,}3{,}6{,}1){,}(7{,}10{,}2{,}8){,}(4{,}9{,}5{,}11)],$\\
$[(0{,}4{,}1{,}10){,}(2{,}5{,}3{,}9){,}(11{,}7{,}6{,}8)],[(0{,}8{,}3{,}7){,}(9{,}6{,}11{,}2){,}(1{,}4{,}10{,}5)],[(0{,}5{,}2{,}1{,}6{,}3{,}4{,}7{,}11{,}10{,}9{,}8)],$\\
 $[(0{,}6{,}2{,}7{,}1{,}5{,}9{,}3{,}10{,}8{,}4{,}11)],[(0{,}7{,}2{,}6{,}4{,}8{,}5{,}10{,}3{,}11{,}1{,}9)],[(0{,}9{,}1{,}7{,}8{,}6{,}10{,}4{,}2{,}11{,}3{,}5)],$\\
$[(0{,}10{,}7{,}5{,}8{,}1{,}11{,}6{,}9{,}4{,}3{,}2)],[(0{,}11{,}9{,}7{,}3{,}1{,}8{,}2{,}10{,}6{,}5{,}4)]$
\item HWP$^*(12;4^6,12^{5})${,}\\
$[(0{,}1{,}2{,}3){,}(4{,}5{,}6{,}7){,}(8{,}9{,}10{,}11)],[(0{,}2{,}4{,}6){,}(8{,}10{,}1{,}3){,}(5{,}7{,}9{,}11)],[(0{,}3{,}6{,}1){,}(7{,}10{,}2{,}8){,}(4{,}9{,}5{,}11)],$\\
$[(0{,}4{,}1{,}10){,}(2{,}5{,}3{,}9){,}(11{,}7{,}6{,}8)],[(0{,}8{,}3{,}7){,}(9{,}6{,}11{,}2){,}(1{,}4{,}10{,}5)],[(0{,}9{,}3{,}2){,}(1{,}5{,}10{,}8){,}(11{,}6{,}4{,}7)],$\\
$[(0{,}5{,}2{,}1{,}6{,}3{,}4{,}11{,}10{,}9{,}7{,}8)],[(0{,}6{,}2{,}7{,}1{,}11{,}3{,}10{,}4{,}8{,}5{,}9)],[(0{,}7{,}2{,}10{,}3{,}11{,}1{,}9{,}8{,}6{,}5{,}4)],$\\ 
$[(0{,}10{,}6{,}9{,}1{,}7{,}3{,}5{,}8{,}4{,}2{,}11)],[(0{,}11{,}9{,}4{,}3{,}1{,}8{,}2{,}6{,}10{,}7{,}5)]$
\item HWP$^*(12;4^{7},12^{4})${,}\\
$[(0{,}3{,}2{,}1){,}(4{,}7{,}6{,}9){,}(5{,}10{,}8{,}11)],[(0{,}4{,}8{,}3){,}(1{,}7{,}9{,}6){,}(2{,}5{,}11{,}10)],[(0{,}5{,}8{,}7){,}(1{,}9{,}2{,}11){,}(3{,}6{,}4{,}10)],$\\
$[(0{,}6{,}2{,}10){,}(1{,}11{,}8{,}5){,}(3{,}9{,}7{,}4)],[(0{,}7{,}5{,}4){,}(1{,}8{,}2{,}9){,}(3{,}10{,}6{,}11)],[(0{,}9{,}5{,}2){,}(1{,}6{,}3{,}8){,}(4{,}11{,}7{,}10)],$\\
$[(0{,}10{,}9{,}8{,}){,}(1{,}4{,}2{,}7){,}(3{,}11{,}6{,}5)],[(0{,}11{,}4{,}1{,}10{,}5{,}9{,}3{,}7{,}2{,}8{,}6)],[(0{,}8{,}4{,}9{,}11{,}2{,}6{,}10{,}7{,}3{,}1{,}5)],$\\
$[(0{,}1{,}2{,}3{,}4{,}5{,}6{,}7{,}8{,}9{,}10{,}11)],[(0{,}2{,}4{,}6{,}8{,}10{,}1{,}3{,}5{,}7{,}11{,}9)]$
\item HWP$^*(12;4^{8},12^{3})${,}\\
$[(0{,}3{,}2{,}1){,}(4{,}7{,}5{,}9){,}(6{,}11{,}10{,}8)],[(0{,}4{,}1{,}6){,}(2{,}5{,}3{,}10){,}(7{,}9{,}8{,}11)],[(0{,}5{,}2{,}7){,}(1{,}10{,}4{,}8){,}(3{,}11{,}6{,}9)],$\\
$[(0{,}6{,}4{,}2){,}(1{,}11{,}8{,}7){,}(3{,}9{,}5{,}10)],[(0{,}7{,}4{,}3){,}(1{,}8{,}5{,}11){,}(2{,}10{,}9{,}6)],[(0{,}9{,}2{,}8){,}(1{,}4{,}11{,}5){,}(3{,}7{,}10{,}6)],$\\
$[(0{,}10{,}5{,}4){,}(1{,}9{,}7{,}6){,}(2{,}11{,}3{,}8)],[(0{,}11{,}4{,}10){,}(1{,}7{,}2{,}9){,}(3{,}6{,}5{,}8)],[(0{,}8{,}4{,}9{,}11{,}2{,}6{,}10{,}7{,}3{,}1{,}5)],$\\
$[(0{,}1{,}2{,}3{,}4{,}5{,}6{,}7{,}8{,}9{,}10{,}11)],[(0{,}2{,}4{,}6{,}8{,}10{,}1{,}3{,}5{,}7{,}11{,}9)]$
\item HWP$^*(12;4^{9},12^{2})${,}\\
$[(0{,}3{,}1{,}4){,}(2{,}5{,}8{,}6){,}(7{,}9{,}11{,}10)],[(0{,}4{,}1{,}5){,}(2{,}6{,}3{,}9){,}(7{,}10{,}8{,}11)],[(0{,}5{,}1{,}6){,}(2{,}7{,}3{,}11){,}(4{,}10{,}9{,}8)],$\\
$[(0{,}6{,}1{,}8){,}(2{,}10{,}3{,}7){,}(4{,}9{,}5{,}11)],[(0{,}7{,}5{,}10){,}(1{,}11{,}8{,}2){,}(3{,}6{,}9{,}4)],[(0{,}8{,}3{,}2){,}(1{,}7{,}4{,}11){,}(5{,}9{,}6{,}10)],$\\
$[(0{,}9{,}7{,}1){,}(2{,}8{,}5{,}4){,}(3{,}10{,}6{,}11)],[(0{,}10{,}4{,}7){,}(1{,}9{,}3{,}8){,}(2{,}11{,}6{,}5)],[(0{,}11{,}5{,}3){,}(1{,}10{,}2{,}9){,}(4{,}8{,}7{,}6)],$\\
$[(0{,}1{,}2{,}3{,}4{,}5{,}6{,}7{,}8{,}9{,}10{,}11)],[(0{,}2{,}4{,}6{,}8{,}10{,}1{,}3{,}5{,}7{,}11{,}9)]$
\item HWP$^*(12;4^{10},12^{1})${,}\\
$[(0{,}2{,}1{,}3){,}(4{,}6{,}5{,}7){,}(8{,}10{,}9{,}11)],[(0{,}3{,}1{,}4){,}(2{,}5{,}8{,}11){,}(6{,}9{,}7{,}10)],[(0{,}4{,}1{,}5){,}(2{,}6{,}3{,}9){,}(7{,}11{,}10{,}8)],$\\
$[(0{,}5{,}1{,}6){,}(2{,}4{,}10{,}7){,}(3{,}11{,}9{,}8)],[(0{,}6{,}2{,}7){,}(1{,}9{,}4{,}11){,}(3{,}8{,}5{,}10)],[(0{,}7{,}1{,}8){,}(2{,}11{,}6{,}10){,}(3{,}5{,}4{,}9)],$\\
$[(0{,}8{,}1{,}10){,}(2{,}9{,}5{,}3){,}(4{,}7{,}6{,}11)],[(0{,}9{,}6{,}1){,}(2{,}10{,}4{,}8){,}(3{,}7{,}5{,}11)],[(0{,}10{,}5{,}2){,}(1{,}11{,}7{,}9){,}(3{,}6{,}8{,}4)],$\\
$[(0{,}11{,}5{,}9){,}(1{,}7{,}3{,}10){,}(2{,}8{,}6{,}4)],[(0{,}1{,}2{,}3{,}4{,}5{,}6{,}7{,}8{,}9{,}10{,}11)]$
\item HWP$^*(12;6^{1},12^{10})${,}\\
$[(0{,}3{,}9{,}4{,}11{,}1){,}(2{,}5{,}10{,}8{,}7{,}6)],[(0{,}1{,}2{,}3{,}4{,}5{,}6{,}7{,}8{,}9{,}10{,}11)],[(0{,}2{,}1{,}3{,}5{,}4{,}6{,}8{,}10{,}7{,}11{,}9)],$\\
$[(0{,}4{,}1{,}5{,}2{,}6{,}3{,}7{,}9{,}8{,}11{,}10)],[(0{,}5{,}1{,}4{,}2{,}7{,}3{,}6{,}10{,}9{,}11{,}8)],[(0{,}6{,}1{,}7{,}10{,}2{,}8{,}3{,}11{,}4{,}9{,}5)],$\\
$[(0{,}7{,}1{,}6{,}9{,}2{,}10{,}4{,}8{,}5{,}11{,}3)],[(0{,}8{,}1{,}9{,}3{,}2{,}4{,}10{,}6{,}11{,}5{,}7)],[(0{,}9{,}7{,}2{,}11{,}6{,}5{,}3{,}10{,}1{,}8{,}4)],$\\
$[(0{,}10{,}3{,}8{,}6{,}4{,}7{,}5{,}9{,}1{,}11{,}2)],[(0{,}11{,}7{,}4{,}3{,}1{,}10{,}5{,}8{,}2{,}9{,}6)]$
\item HWP$^*(12;6^{2},12^{9})${,}\\
$[(0{,}3{,}9{,}4{,}11{,}1){,}(2{,}5{,}10{,}8{,}7{,}6)],[(0{,}4{,}7{,}1{,}9{,}2){,}(3{,}10{,}6{,}11{,}8{,}5)],[( 0{,}1{,}2{,}3{,}4{,}5{,}6{,}7{,}8{,}9{,}10{,}11)],$\\
$[(0{,}2{,}1{,}3{,}5{,}4{,}6{,}8{,}10{,}7{,}11{,}9)],[(0{,}5{,}1{,}4{,}2{,}6{,}3{,}7{,}9{,}8{,}11{,}10)],[(0{,}6{,}1{,}5{,}2{,}4{,}9{,}11{,}7{,}10{,}3{,}8)],$\\
$[(0{,}7{,}2{,}8{,}1{,}6{,}4{,}10{,}9{,}3{,}11{,}5)],[(0{,}8{,}2{,}7{,}3{,}6{,}9{,}1{,}10{,}5{,}11{,}4)],[(0{,}9{,}5{,}7{,}4{,}1{,}8{,}6{,}10{,}2{,}11{,}3)],$\\
$[(0{,}10{,}1{,}11{,}6{,}5{,}8{,}4{,}3{,}2{,}9{,}7)],[(0{,}11{,}2{,}10{,}4{,}8{,}3{,}1{,}7{,}5{,}9{,}6)]$
\item HWP$^*(12;6^{3},12^{8})${,}\\
$[(0{,}3{,}9{,}4{,}11{,}1){,}(2{,}5{,}10{,}8{,}7{,}6)],[(0{,}4{,}7{,}1{,}9{,}2){,}(3{,}10{,}6{,}11{,}8{,}5)],[(0{,}6{,}1{,}11{,}3{,}8){,}(2{,}7{,}10{,}9{,}5{,}4)],$\\
$[(0{,}1{,}2{,}3{,}4{,}5{,}6{,}7{,}8{,}9{,}10{,}11)],[(0{,}2{,}1{,}3{,}5{,}7{,}4{,}6{,}9{,}8{,}11{,}10)],[(0{,}5{,}1{,}4{,}3{,}2{,}6{,}8{,}10{,}7{,}11{,}9)],$\\
$[(0{,}7{,}2{,}4{,}8{,}1{,}5{,}9{,}11{,}6{,}10{,}3)],[(0{,}8{,}2{,}9{,}3{,}6{,}5{,}11{,}4{,}10{,}1{,}7)],[(0{,}9{,}1{,}8{,}6{,}3{,}11{,}7{,}5{,}2{,}10{,}4)],$\\
$[(0{,}10{,}2{,}11{,}5{,}8{,}4{,}9{,}7{,}3{,}1{,}6)],[(0{,}11{,}2{,}8{,}3{,}7{,}9{,}6{,}4{,}1{,}10{,}5)]$
\item HWP$^*(12;6^{4},12^{7})${,}\\
$[(0{,}3{,}9{,}4{,}11{,}1){,}(2{,}5{,}10{,}8{,}7{,}6)],[(0{,}4{,}7{,}1{,}9{,}2){,}(3{,}10{,}6{,}11{,}8{,}5)],[(0{,}6{,}1{,}11{,}3{,}8){,}(2{,}7{,}10{,}9{,}5{,}4)],$\\
$[(0{,}9{,}1{,}8{,}2{,}10){,}(3{,}6{,}5{,}11{,}7{,}4)],[(0{,}1{,}2{,}3{,}4{,}5{,}6{,}7{,}8{,}9{,}10{,}11)],[(0{,}2{,}1{,}3{,}5{,}7{,}9{,}6{,}8{,}11{,}10{,}4)],$\\
$[(0{,}5{,}1{,}4{,}6{,}3{,}2{,}8{,}10{,}7{,}11{,}9)],[(0{,}7{,}2{,}4{,}1{,}10{,}5{,}9{,}8{,}3{,}11{,}6)],[(0{,}8{,}4{,}9{,}11{,}2{,}6{,}10{,}3{,}1{,}7{,}5)],$\\
$[(0{,}10{,}1{,}5{,}2{,}11{,}4{,}8{,}6{,}9{,}3{,}7)],[(0{,}11{,}5{,}8{,}1{,}6{,}4{,}10{,}2{,}9{,}7{,}3)]$
\item HWP$^*(12;6^{5},12^{6})${,}\\
$[(0{,}3{,}6{,}1{,}9{,}2){,}(4{,}11{,}7{,}10{,}8{,}5)],[(0{,}4{,}2{,}5{,}10{,}3){,}(1{,}11{,}8{,}7{,}6{,}9)],[(0{,}6{,}11{,}1{,}4{,}10){,}(2{,}7{,}5{,}3{,}9{,}8)],$\\
$[(0{,}7{,}4{,}3{,}8{,}1){,}(2{,}10{,}9{,}5{,}11{,}6)],[(0{,}9{,}4{,}7{,}1{,}8){,}(2{,}11{,}3{,}10{,}6{,}5)]{,}[(0{,}1{,}2{,}3{,}4{,}5{,}6{,}7{,}8{,}9{,}10{,}11)],$\\
$[(0{,}2{,}4{,}6{,}8{,}10{,}1{,}3{,}5{,}7{,}11{,}9)],[(0{,}5{,}8{,}3{,}11{,}10{,}2{,}1{,}7{,}9{,}6{,}4)],[(0{,}8{,}4{,}9{,}11{,}2{,}6{,}10{,}7{,}3{,}1{,}5)],$\\
$[(0{,}10{,}4{,}8{,}11{,}5{,}1{,}6{,}3{,}2{,}9{,}7)],[(0{,}11{,}4{,}1{,}10{,}5{,}9{,}3{,}7{,}2{,}8{,}6)]$
\item HWP$^*(12;6^{6},12^{5})${,}\\
$[(0{,}3{,}2{,}5{,}1{,}8){,}(4{,}11{,}6{,}9{,}7{,}10)],[(0{,}4{,}2{,}7{,}5{,}10){,}(1{,}6{,}11{,}3{,}9{,}8)],[(0{,}6{,}1{,}9{,}5{,}2){,}(3{,}8{,}11{,}7{,}4{,}10)],$\\
$[(0{,}7{,}6{,}5{,}11{,}1){,}(2{,}9{,}4{,}3{,}10{,}8)],[(0{,}9{,}2{,}10{,}6{,}3){,}(1{,}11{,}5{,}4{,}8{,}7)],[(0{,}10{,}9{,}1{,}4{,}7){,}(2{,}11{,}8{,}5{,}3{,}6)],$\\
$[(0{,}1{,}2{,}3{,}4{,}5{,}6{,}7{,}8{,}9{,}10{,}11)],[(0{,}2{,}4{,}6{,}8{,}10{,}1{,}3{,}5{,}7{,}11{,}9)],[(0{,}5{,}8{,}3{,}11{,}10{,}2{,}1{,}7{,}9{,}6{,}4)],$\\
$[(0{,}8{,}4{,}9{,}11{,}2{,}6{,}10{,}7{,}3{,}1{,}5)],[(0{,}11{,}4{,}1{,}10{,}5{,}9{,}3{,}7{,}2{,}8{,}6)]$
\item HWP$^*(12;6^{7},12^{4})${,}\\
$[(0{,}3{,}2{,}5{,}1{,}6){,}(4{,}7{,}10{,}9{,}8{,}11)],[(0{,}4{,}1{,}8{,}2{,}7){,}(3{,}9{,}5{,}10{,}6{,}11)],[(0{,}6{,}1{,}9{,}3{,}10){,}(2{,}8{,}5{,}4{,}11{,}7)],$\\
$[(0{,}7{,}5{,}11{,}8{,}1){,}(2{,}9{,}4{,}10{,}3{,}6)],[(0{,}9{,}2{,}10{,}4{,}8){,}(1{,}11{,}6{,}5{,}3{,}7)],[(0{,}10{,}8{,}7{,}6{,}3){,}(1{,}4{,}2{,}11{,}5{,}9)],$\\
$[(0{,}11{,}1{,}10{,}5{,}2){,}(3{,}8{,}6{,}9{,}7{,}4)],[(0{,}1{,}2{,}3{,}4{,}5{,}6{,}7{,}8{,}9{,}10{,}11)],[(0{,}2{,}4{,}6{,}8{,}10{,}1{,}3{,}5{,}7{,}11{,}9)]{,}$\\
$[(0{,}5{,}8{,}3{,}11{,}10{,}2{,}1{,}7{,}9{,}6{,}4)],[(0{,}8{,}4{,}9{,}11{,}2{,}6{,}10{,}7{,}3{,}1{,}5)]$
\item HWP$^*(12;6^{8},12^{3})${,}\\
$[(0{,}3{,}1{,}4{,}2{,}5){,}(6{,}9{,}7{,}10{,}8{,}11)],[(0{,}4{,}1{,}5{,}2{,}6){,}(3{,}10{,}9{,}11{,}8{,}7)],[(0{,}6{,}1{,}8{,}2{,}7){,}(3{,}9{,}4{,}10{,}5{,}11)],$\\
$[(0{,}7{,}1{,}6{,}10{,}3){,}(2{,}11{,}4{,}8{,}5{,}9)],[(0{,}8{,}4{,}7{,}2{,}10){,}(1{,}9{,}5{,}3{,}6{,}11)],[(0{,}9{,}3{,}8{,}6{,}2){,}(1{,}10{,}7{,}4{,}11{,}5)],$\\
$[(0{,}10{,}6{,}3{,}2{,}8){,}(1{,}11{,}7{,}5{,}4{,}9)],[(0{,}11{,}2{,}9{,}8{,}1){,}(3{,}7{,}6{,}5{,}10{,}4)],[(0{,}1{,}2{,}3{,}4{,}5{,}6{,}7{,}8{,}9{,}10{,}11)],$\\
$[(0{,}2{,}4{,}6{,}8{,}10{,}1{,}3{,}5{,}7{,}11{,}9)],[(0{,}5{,}8{,}3{,}11{,}10{,}2{,}1{,}7{,}9{,}6{,}4)]$
\item HWP$^*(12;6^{9},12^{2})${,}\\
$[(0{,}3{,}1{,}4{,}2{,}5){,}(6{,}9{,}7{,}10{,}8{,}11)],[(0{,}4{,}1{,}5{,}2{,}6){,}(3{,}8{,}7{,}9{,}11{,}10)],[(0{,}5{,}1{,}6{,}2{,}7){,}(3{,}10{,}9{,}4{,}11{,}8)],$\\
$[(0{,}6{,}1{,}7{,}2{,}8){,}(3{,}11{,}4{,}10{,}5{,}9)],[(0{,}7{,}1{,}9{,}6{,}3){,}(2{,}11{,}5{,}10{,}4{,}8)],[(0{,}8{,}5{,}3{,}6{,}10){,}(1{,}11{,}7{,}4{,}9{,}2)],$\\
$[(0{,}9{,}8{,}6{,}11{,}1){,}(2{,}10{,}7{,}5{,}4{,}3)],[(0{,}10{,}6{,}5{,}11{,}2){,}(1{,}8{,}4{,}7{,}3{,}9)],[(0{,}11{,}3{,}7{,}6{,}4){,}(1{,}10{,}2{,}9{,}5{,}8)],$\\
$[(0{,}1{,}2{,}3{,}4{,}5{,}6{,}7{,}8{,}9{,}10{,}11)],[(0{,}2{,}4{,}6{,}8{,}10{,}1{,}3{,}5{,}7{,}11{,}9)]$
\item HWP$^*(12;6^{10},12^{1})${,}\\
$[(0{,}2{,}1{,}3{,}5{,}4){,}(6{,}8{,}7{,}9{,}11{,}10)],[(0{,}3{,}1{,}4{,}2{,}5){,}(6{,}9{,}7{,}10{,}8{,}11)],[(0{,}4{,}1{,}5{,}2{,}6){,}(3{,}7{,}11{,}8{,}10{,}9)],$\\
$[(0{,}5{,}1{,}6{,}2{,}7){,}(3{,}10{,}4{,}11{,}9{,}8)],[(0{,}6{,}1{,}7{,}2{,}8){,}(3{,}9{,}4{,}10{,}5{,}11)],[(0{,}7{,}1{,}9{,}5{,}10){,}(2{,}11{,}4{,}8{,}6{,}3)],$\\
$[(0{,}8{,}1{,}10{,}2{,}9){,}(3{,}11{,}7{,}4{,}6{,}5)],[(0{,}9{,}6{,}4{,}7{,}3){,}(1{,}11{,}5{,}8{,}2{,}10)],[(0{,}10{,}7{,}5{,}9{,}2){,}(1{,}8{,}4{,}3{,}6{,}11)],$\\
$[(0{,}11{,}2{,}4{,}9{,}1){,}(3{,}8{,}5{,}7{,}6{,}10)],[(0{,}1{,}2{,}3{,}4{,}5{,}6{,}7{,}8{,}9{,}10{,}11)]$
\item HWP$^*(16;8^{2},16^{13})${,}\\
$[(0{,}1{,}2{,}3{,}4{,}5{,}6{,}7){,}(8{,}9{,}10{,}11{,}12{,}13{,}14{,}15)],[(0{,}2{,}4{,}6{,}8{,}10{,}12{,}14){,}(1{,}3{,}5{,}7{,}9{,}11{,}15{,}13)]{,}$\\
$[(0{,}3{,}8{,}6{,}15{,}10{,}1{,}14{,}5{,}4{,}9{,}7{,}13{,}2{,}12{,}11)],[(0{,}4{,}1{,}6{,}2{,}5{,}8{,}3{,}7{,}10{,}9{,}12{,}15{,}11{,}14{,}13)]{,}$\\
$[(0{,}5{,}1{,}7{,}3{,}2{,}6{,}4{,}8{,}11{,}9{,}14{,}10{,}13{,}15{,}12)],[(0{,}6{,}1{,}9{,}3{,}10{,}2{,}7{,}4{,}11{,}13{,}8{,}14{,}12{,}5{,}15)]{,}$\\
$[(0{,}7{,}1{,}10{,}3{,}6{,}5{,}2{,}15{,}14{,}8{,}13{,}11{,}4{,}12{,}9)],[(0{,}8{,}2{,}1{,}11{,}5{,}10{,}14{,}3{,}12{,}7{,}15{,}9{,}6{,}13{,}4)]{,}$\\
$[(0{,}9{,}4{,}2{,}8{,}12{,}3{,}13{,}6{,}14{,}1{,}15{,}7{,}11{,}10{,}5)],[(0{,}10{,}15{,}6{,}12{,}1{,}4{,}3{,}14{,}11{,}8{,}5{,}9{,}13{,}7{,}2)]{,}$\\
$[(0{,}11{,}3{,}15{,}2{,}13{,}5{,}12{,}4{,}14{,}9{,}1{,}8{,}7{,}6{,}10)],[(0{,}12{,}2{,}14{,}7{,}5{,}13{,}10{,}6{,}9{,}8{,}4{,}15{,}3{,}11{,}1)]{,}$\\
$[(0{,}13{,}3{,}1{,}12{,}6{,}11{,}2{,}9{,}15{,}5{,}14{,}4{,}10{,}7{,}8)],[(0{,}14{,}2{,}10{,}4{,}7{,}12{,}8{,}15{,}1{,}13{,}9{,}5{,}11{,}6{,}3)]{,}$\\
$[(0{,}15{,}4{,}13{,}12{,}10{,}8{,}1{,}5{,}3{,}9{,}2{,}11{,}7{,}14{,}6)]$
\item HWP$^*(16;8^{4},16^{11})${,}\\
$[(0{,}6{,}12{,}15{,}5{,}14{,}7{,}1{,}){,}(2{,}9{,}4{,}3{,}13{,}10{,}8{,}11)],[(0{,}9{,}15{,}12{,}6{,}11{,}4{,}7){,}(1{,}10{,}2{,}5{,}13{,}8{,}14{,}3)]{,}$\\
$[(0{,}12{,}3{,}2{,}7{,}14{,}8{,}4){,}(1{,}11{,}10{,}5{,}15{,}6{,}13{,}9)],[(0{,}13{,}11{,}1{,}15{,}9{,}12{,}5){,}(2{,}14{,}6{,}10{,}3{,}7{,}4{,}8)]{,}$\\
$[(0{,}1{,}2{,}3{,}4{,}5{,}6{,}7{,}8{,}9{,}10{,}11{,}12{,}13{,}14{,}15)],[(0{,}2{,}4{,}6{,}8{,}10{,}12{,}14{,}1{,}3{,}5{,}7{,}9{,}11{,}15{,}13)]{,}$\\
$[(0{,}3{,}8{,}6{,}15{,}10{,}1{,}14{,}5{,}4{,}9{,}7{,}13{,}2{,}12{,}11)],[(0{,}4{,}2{,}1{,}8{,}3{,}6{,}5{,}11{,}9{,}14{,}10{,}13{,}15{,}7{,}12)]{,}$\\
$[(0{,}5{,}2{,}6{,}1{,}9{,}3{,}10{,}7{,}11{,}13{,}4{,}15{,}14{,}12{,}8)],[(0{,}7{,}5{,}1{,}4{,}10{,}14{,}13{,}3{,}12{,}2{,}8{,}15{,}11{,}6{,}9)]{,}$\\
$[(0{,}8{,}1{,}7{,}15{,}3{,}11{,}14{,}4{,}13{,}5{,}12{,}9{,}6{,}2{,}10)],[(0{,}10{,}15{,}4{,}12{,}1{,}6{,}3{,}14{,}11{,}8{,}5{,}9{,}13{,}7{,}2)]{,}$\\
$[(0{,}11{,}3{,}15{,}2{,}13{,}1{,}12{,}4{,}14{,}9{,}5{,}8{,}7{,}10{,}6)],[(0{,}14{,}2{,}15{,}1{,}13{,}6{,}4{,}11{,}5{,}10{,}9{,}8{,}12{,}7{,}3)]{,}$\\
$[(0{,}15{,}8{,}13{,}12{,}10{,}4{,}1{,}5{,}3{,}9{,}2{,}11{,}7{,}6{,}14)]$
\item HWP$^*(16;8^{6},16^{9})${,}\\
$[(0{,}4{,}2{,}1{,}8{,}3{,}6{,}5){,}(7{,}11{,}10{,}13{,}9{,}12{,}15{,}14)],[(0{,}5{,}2{,}6{,}1{,}9{,}15{,}12){,}(3{,}10{,}7{,}14{,}8{,}11{,}13{,}4)]{,}$\\
$[(0{,}6{,}11{,}1{,}10{,}8{,}4{,}7){,}(2{,}14{,}12{,}5{,}13{,}15{,}9{,}3)],[(0{,}9{,}14{,}10{,}3{,}1{,}11{,}4){,}(2{,}5{,}15{,}7{,}12{,}6{,}13{,}8)]{,}$\\
$[(0{,}12{,}3{,}13{,}11{,}9{,}4{,}8){,}(1{,}15{,}5{,}14{,}6{,}10{,}2{,}7)],[(0{,}13{,}10{,}5{,}11{,}2{,}9{,}1){,}(3{,}7{,}4{,}15{,}6{,}12{,}8{,}14)]{,}$\\
$[(0{,}1{,}2{,}3{,}4{,}5{,}6{,}7{,}8{,}9{,}10{,}11{,}12{,}13{,}14{,}15)],[(0{,}2{,}4{,}6{,}8{,}10{,}12{,}14{,}1{,}3{,}5{,}7{,}9{,}11{,}15{,}13)]{,}$\\
$[(0{,}3{,}8{,}6{,}15{,}10{,}1{,}14{,}5{,}4{,}9{,}7{,}13{,}2{,}12{,}11)],[(0{,}7{,}5{,}1{,}4{,}10{,}14{,}13{,}3{,}12{,}2{,}8{,}15{,}11{,}6{,}9)]{,}$\\
$[(0{,}8{,}1{,}7{,}15{,}3{,}11{,}14{,}4{,}13{,}5{,}12{,}9{,}6{,}2{,}10)],[(0{,}10{,}15{,}4{,}12{,}1{,}6{,}3{,}14{,}11{,}8{,}5{,}9{,}13{,}7{,}2)]{,}$\\
$[(0{,}11{,}3{,}15{,}2{,}13{,}1{,}12{,}4{,}14{,}9{,}5{,}8{,}7{,}10{,}6)],[(0{,}14{,}2{,}15{,}1{,}13{,}6{,}4{,}11{,}5{,}10{,}9{,}8{,}12{,}7{,}3)]{,}$\\
$[(0{,}15{,}8{,}13{,}12{,}10{,}4{,}1{,}5{,}3{,}9{,}2{,}11{,}7{,}6{,}14)]$
\item HWP$^*(16;8^{8},16^{7})${,}\\
$[( 0{,}4{,}2{,}1{,}7{,}3{,}6{,}5){,}(8{,}11{,}9{,}12{,}15{,}14{,}10{,}13)],[(0{,}5{,}2{,}6{,}1{,}8{,}3{,}10){,}(4{,}7{,}11{,}13{,}15{,}12{,}9{,}14)]{,}$\\
$[(0{,}6{,}2{,}5{,}10{,}3{,}7{,}12){,}(1{,}11{,}14{,}8{,}4{,}13{,}9{,}15)],[(0{,}8{,}1{,}9{,}3{,}2{,}10{,}7){,}(4{,}15{,}5{,}14{,}12{,}6{,}13{,}11)]{,}$\\
$[(0{,}9{,}8{,}12{,}7{,}14{,}6{,}4){,}(1{,}10{,}2{,}15{,}3{,}13{,}5{,}11)],[(0{,}12{,}8{,}14{,}2{,}9{,}4{,}3){,}(1{,}13{,}6{,}11{,}10{,}5{,}15{,}7)]{,}$\\
$[(0{,}13{,}4{,}11{,}5{,}12{,}3{,}1){,}(2{,}14{,}7{,}15{,}9{,}6{,}10{,}8)],[(0{,}14{,}3{,}11{,}2{,}7{,}4{,}8){,}(1{,}15{,}6{,}12{,}5{,}13{,}10{,}9)]{,}$\\
$[(0{,}1{,}2{,}3{,}4{,}5{,}6{,}7{,}8{,}9{,}10{,}11{,}12{,}13{,}14{,}15)],[(0{,}2{,}4{,}6{,}8{,}10{,}12{,}14{,}1{,}3{,}5{,}7{,}9{,}11{,}15{,}13)]{,}$\\
$[(0{,}3{,}8{,}6{,}15{,}10{,}1{,}14{,}5{,}4{,}9{,}7{,}13{,}2{,}12{,}11)],[(0{,}7{,}5{,}1{,}4{,}10{,}14{,}13{,}3{,}12{,}2{,}8{,}15{,}11{,}6{,}9)]{,}$\\
$[(0{,}10{,}15{,}4{,}12{,}1{,}6{,}3{,}14{,}11{,}8{,}5{,}9{,}13{,}7{,}2)],[(0{,}11{,}3{,}15{,}2{,}13{,}1{,}12{,}4{,}14{,}9{,}5{,}8{,}7{,}10{,}6)]{,}$\\
$[(0{,}15{,}8{,}13{,}12{,}10{,}4{,}1{,}5{,}3{,}9{,}2{,}11{,}7{,}6{,}14)]$
\item HWP$^*(16;8^{10},16^{5})${,}\\
$[(0{,}4{,}2{,}1{,}6{,}3{,}7{,}5){,}(8{,}11{,}9{,}12{,}15{,}14{,}10{,}13)],[(0{,}5{,}1{,}4{,}3{,}2{,}6{,}9){,}(7{,}11{,}8{,}14{,}13{,}10{,}15{,}12)]{,}$\\
$[(0{,}6{,}1{,}7{,}2{,}5{,}9{,}3){,}(4{,}10{,}14{,}12{,}8{,}15{,}11{,}13)],[(0{,}7{,}1{,}8{,}2{,}9{,}4{,}12){,}(3{,}6{,}13{,}15{,}5{,}14{,}11{,}10)]$\\
$[(0{,}8{,}1{,}9{,}6{,}2{,}10{,}7){,}(3{,}12{,}5{,}13{,}11{,}14{,}4{,}15)],[(0{,}9{,}1{,}10{,}2{,}7{,}14{,}8){,}(3{,}11{,}5{,}15{,}4{,}13{,}6{,}12)]{,}$\\
$[(0{,}10{,}8{,}3{,}13{,}9{,}15{,}1){,}(2{,}14{,}7{,}4{,}11{,}6{,}5{,}12)],[(0{,}12{,}1{,}15{,}9{,}13{,}5{,}10){,}(2{,}8{,}4{,}7{,}3{,}14{,}6{,}11)]{,}$\\
$[(0{,}13{,}7{,}15{,}6{,}10{,}5{,}2){,}(1{,}11{,}4{,}8{,}12{,}9{,}14{,}3)],[(0{,}14{,}2{,}15{,}7{,}12{,}6{,}4){,}(1{,}13{,}3{,}10{,}9{,}8{,}5{,}11)]{,}$\\
$[(0{,}11{,}3{,}15{,}2{,}13{,}1{,}12{,}4{,}14{,}9{,}5{,}8{,}7{,}10{,}6)],[(0{,}15{,}8{,}13{,}12{,}10{,}4{,}1{,}5{,}3{,}9{,}2{,}11{,}7{,}6{,}14)]{,}$\\
$[(0{,}3{,}8{,}6{,}15{,}10{,}1{,}14{,}5{,}4{,}9{,}7{,}13{,}2{,}12{,}11)],[(0{,}2{,}4{,}6{,}8{,}10{,}12{,}14{,}1{,}3{,}5{,}7{,}9{,}11{,}15{,}13)]{,}$\\ 
$[(0{,}1{,}2{,}3{,}4{,}5{,}6{,}7{,}8{,}9{,}10{,}11{,}12{,}13{,}14{,}15)]$
\item HWP$^*(16;8^{12},16^{3})${,}\\
$[(0{,}4{,}1{,}5{,}2{,}6{,}3{,}7){,}(8{,}11{,}9{,}12{,}10{,}13{,}15{,}14)],[(0{,}5{,}1{,}4{,}2{,}7{,}3{,}6){,}(8{,}12{,}9{,}13{,}11{,}14{,}10{,}15)]{,}$\\
$[(0{,}6{,}1{,}7{,}2{,}5{,}3{,}9){,}(4{,}8{,}13{,}10{,}14{,}12{,}15{,}11)],[(0{,}7{,}1{,}6{,}2{,}8{,}3{,}10){,}(4{,}11{,}5{,}14{,}13{,}9{,}15{,}12)]{,}$\\
$[(0{,}8{,}1{,}9{,}2{,}10{,}3{,}12){,}(4{,}7{,}5{,}15{,}6{,}14{,}11{,}13)],[(0{,}9{,}1{,}8{,}2{,}11{,}3{,}14){,}(4{,}10{,}5{,}12{,}6{,}13{,}7{,}15)]{,}$\\
$[(0{,}10{,}2{,}1{,}11{,}6{,}4{,}3){,}(5{,}13{,}8{,}15{,}9{,}14{,}7{,}12)],[(0{,}11{,}1{,}10{,}4{,}13{,}12{,}2){,}(3{,}15{,}7{,}6{,}5{,}9{,}8{,}14)]{,}$\\
$[(0{,}12{,}8{,}4{,}14{,}2{,}15{,}1){,}(3{,}13{,}5{,}11{,}7{,}10{,}6{,}9)],[(0{,}13{,}6{,}10{,}7{,}11{,}8{,}5){,}(1{,}15{,}2{,}14{,}9{,}4{,}12{,}3)]{,}$\\ 
$[(0{,}14{,}6{,}11{,}10{,}9{,}5{,}8){,}(1{,}12{,}7{,}4{,}15{,}3{,}2{,}13)],[(0{,}15{,}5{,}10{,}8{,}7{,}14{,}4){,}(1{,}13{,}3{,}11{,}2{,}9{,}6{,}12)]{,}$\\
$[(0{,}3{,}8{,}6{,}15{,}10{,}1{,}14{,}5{,}4{,}9{,}7{,}13{,}2{,}12{,}11)],[(0{,}2{,}4{,}6{,}8{,}10{,}12{,}14{,}1{,}3{,}5{,}7{,}9{,}11{,}15{,}13)]{,}$\\
$[(0{,}1{,}2{,}3{,}4{,}5{,}6{,}7{,}8{,}9{,}10{,}11{,}12{,}13{,}14{,}15)]$
\item HWP$^*(16;8^{14},16^{1})${,}\\
$[(0{,}2{,}1{,}3{,}5{,}4{,}6{,}8){,}(7{,}9{,}11{,}10{,}12{,}14{,}13{,}15)],[(0{,}3{,}1{,}4{,}2{,}5{,}7{,}6){,}(8{,}10{,}9{,}12{,}15{,}13{,}11{,}14)]{,}$\\
$[(0{,}4{,}1{,}5{,}2{,}6{,}3{,}7){,}(8{,}11{,}9{,}13{,}10{,}15{,}14{,}12)],[(0{,}5{,}1{,}6{,}2{,}4{,}3{,}9){,}(7{,}10{,}13{,}8{,}14{,}11{,}15{,}12)]{,}$\\
$[(0{,}6{,}1{,}7{,}2{,}8{,}3{,}10){,}(4{,}9{,}14{,}5{,}15{,}11{,}13{,}12)],[(0{,}7{,}1{,}8{,}2{,}9{,}3{,}11){,}(4{,}12{,}5{,}13{,}6{,}15{,}10{,}14)]{,}$\\
$[(0{,}8{,}1{,}9{,}2{,}7{,}3{,}13){,}(4{,}15{,}5{,}12{,}11{,}6{,}14{,}10)],[(0{,}9{,}1{,}10{,}2{,}11{,}3{,}12){,}(4{,}7{,}14{,}6{,}13{,}5{,}8{,}15)]{,}$\\
$[(0{,}10{,}1{,}11{,}2{,}12{,}3{,}14){,}(4{,}8{,}13{,}7{,}5{,}9{,}15{,}6)],[(0{,}11{,}1{,}12{,}2{,}13{,}9{,}5){,}(3{,}15{,}8{,}6{,}10{,}7{,}4{,}14)]{,}$\\
$[(0{,}12{,}9{,}6{,}11{,}7{,}13{,}1){,}(2{,}15{,}3{,}8{,}4{,}10{,}5{,}14)],[(0{,}13{,}3{,}2{,}14{,}7{,}11{,}4){,}(1{,}15{,}9{,}8{,}5{,}10{,}6{,}12)]{,}$\\
$[(0{,}14{,}1{,}13{,}4{,}11{,}5{,}3){,}(2{,}10{,}8{,}12{,}6{,}9{,}7{,}15)],[(0{,}15{,}1{,}14{,}9{,}4{,}13{,}2){,}(3{,}6{,}5{,}11{,}8{,}7{,}12{,}10)]{,}$\\ 
$[(0{,}1{,}2{,}3{,}4{,}5{,}6{,}7{,}8{,}9{,}10{,}11{,}12{,}13{,}14{,}15)]$
\item HWP$^*(16;4^{1},16^{14})${,}\\
$[(0{,}1{,}2{,}3){,}(4{,}5{,}6{,}7){,}(8{,}9{,}10{,}11){,}(12{,}13{,}14{,}15)],[(0{,}2{,}4{,}3{,}7{,}1{,}8{,}5{,}9{,}6{,}10{,}12{,}15{,}14{,}11{,}13)]{,}$\\
$[(0{,}3{,}9{,}12{,}1{,}5{,}13{,}8{,}2{,}6{,}14{,}7{,}11{,}15{,}10{,}4)],[(0{,}4{,}12{,}2{,}1{,}11{,}7{,}5{,}8{,}14{,}9{,}13{,}10{,}3{,}6{,}15)]{,}$\\
$[(0{,}5{,}12{,}4{,}1{,}15{,}3{,}11{,}2{,}14{,}6{,}8{,}7{,}10{,}13{,}9)],[(0{,}6{,}11{,}5{,}10{,}7{,}12{,}9{,}2{,}15{,}4{,}8{,}3{,}13{,}1{,}14)]{,}$\\
$[(0{,}7{,}2{,}8{,}1{,}4{,}6{,}3{,}10{,}5{,}15{,}9{,}14{,}13{,}11{,}12)],[(0{,}8{,}4{,}15{,}13{,}2{,}12{,}14{,}1{,}3{,}5{,}7{,}9{,}11{,}10{,}6)]{,}$\\
$[(0{,}9{,}15{,}8{,}11{,}14{,}12{,}10{,}2{,}5{,}3{,}1{,}6{,}13{,}4{,}7)],[(0{,}10{,}1{,}7{,}3{,}2{,}9{,}4{,}13{,}5{,}14{,}8{,}15{,}6{,}12{,}11)]{,}$\\
$[(0{,}11{,}1{,}9{,}3{,}4{,}2{,}13{,}15{,}7{,}14{,}10{,}8{,}12{,}6{,}5)],[(0{,}12{,}3{,}8{,}6{,}1{,}13{,}7{,}15{,}11{,}9{,}5{,}4{,}14{,}2{,}10)]{,}$\\
$[(0{,}13{,}3{,}14{,}5{,}11{,}4{,}9{,}8{,}10{,}15{,}1{,}12{,}7{,}6{,}2)],[(0{,}14{,}3{,}15{,}2{,}11{,}6{,}4{,}10{,}9{,}7{,}8{,}13{,}12{,}5{,}1)]{,}$\\
$[(0{,}15{,}5{,}2{,}7{,}13{,}6{,}9{,}1{,}10{,}14{,}4{,}11{,}3{,}12{,}8)]$
\item HWP$^*(16;4^{2},16^{13})${,}\\
$[(0{,}1{,}2{,}3){,}(4{,}5{,}6{,}7){,}(8{,}9{,}10{,}11){,}(12{,}13{,}14{,}15)],[(0{,}2{,}4{,}6){,}(8{,}10{,}12{,}14){,}(1{,}3{,}5{,}7){,}(9{,}11{,}13{,}15)],$\\
$[(0{,}3{,}9{,}12{,}1{,}5{,}13{,}8{,}2{,}6{,}14{,}7{,}11{,}15{,}10{,}4)],[(0{,}4{,}12{,}2{,}1{,}11{,}7{,}5{,}8{,}14{,}9{,}13{,}10{,}3{,}6{,}15)],$\\
$[(0{,}5{,}12{,}4{,}1{,}15{,}3{,}11{,}2{,}14{,}6{,}8{,}7{,}10{,}13{,}9)],[(0{,}6{,}11{,}5{,}10{,}7{,}12{,}9{,}2{,}15{,}4{,}8{,}3{,}13{,}1{,}14)],$\\
$[(0{,}7{,}2{,}8{,}1{,}4{,}3{,}10{,}5{,}9{,}6{,}12{,}15{,}14{,}13{,}11)],[(0{,}8{,}4{,}2{,}7{,}3{,}12{,}5{,}1{,}9{,}14{,}11{,}6{,}10{,}15{,}13)],$\\
$[(0{,}9{,}15{,}8{,}11{,}14{,}12{,}10{,}2{,}5{,}3{,}1{,}6{,}13{,}4{,}7)],[(0{,}10{,}1{,}7{,}6{,}2{,}9{,}3{,}8{,}13{,}5{,}14{,}4{,}15{,}11{,}12)],$\\
$[(0{,}11{,}1{,}8{,}5{,}2{,}12{,}3{,}15{,}6{,}9{,}4{,}13{,}7{,}14{,}10)],[(0{,}12{,}6{,}1{,}10{,}8{,}15{,}2{,}11{,}4{,}9{,}7{,}13{,}3{,}14{,}5)],$\\
$[(0{,}13{,}2{,}10{,}6{,}3{,}4{,}14{,}1{,}12{,}11{,}9{,}5{,}15{,}7{,}8)],[(0{,}14{,}3{,}2{,}13{,}6{,}5{,}4{,}11{,}10{,}9{,}8{,}12{,}7{,}15{,}1)],$\\
$[(0{,}15{,}5{,}11{,}3{,}7{,}9{,}1{,}13{,}12{,}8{,}6{,}4{,}10{,}14{,}2)]$
\item HWP$^*(16;4^{3},16^{12})${,}\\
$[(0{,}1{,}2{,}3){,}(4{,}5{,}6{,}7){,}(8{,}9{,}10{,}11){,}(12{,}13{,}14{,}15)],[(0{,}2{,}4{,}6){,}(8{,}10{,}12{,}14){,}(1{,}3{,}5{,}7){,}(9{,}11{,}13{,}15)],$\\
$[(0{,}3{,}9{,}12){,}(1{,}5{,}13{,}8){,}(2{,}6{,}14{,}7){,}(10{,}15{,}11{,}4)],[(0{,}4{,}12{,}2{,}1{,}11{,}7{,}5{,}8{,}14{,}9{,}13{,}10{,}3{,}6{,}15)],$\\
$[(0{,}5{,}12{,}4{,}1{,}15{,}3{,}11{,}2{,}14{,}6{,}8{,}7{,}10{,}13{,}9)],[(0{,}6{,}11{,}5{,}10{,}7{,}12{,}9{,}2{,}15{,}4{,}8{,}3{,}13{,}1{,}14)],$\\
$[(0{,}7{,}3{,}2{,}8{,}4{,}9{,}1{,}10{,}5{,}11{,}6{,}12{,}15{,}14{,}13)],[(0{,}8{,}2{,}7{,}6{,}1{,}4{,}3{,}10{,}9{,}14{,}5{,}15{,}13{,}12{,}11)],$\\
$[(0{,}9{,}15{,}8{,}11{,}14{,}12{,}10{,}2{,}5{,}3{,}1{,}6{,}13{,}4{,}7)],[(0{,}10{,}1{,}7{,}8{,}5{,}14{,}2{,}13{,}11{,}12{,}3{,}15{,}6{,}9{,}4)],$\\
$[(0{,}11{,}1{,}8{,}6{,}2{,}9{,}3{,}12{,}7{,}15{,}10{,}14{,}4{,}13{,}5)],[(0{,}12{,}1{,}9{,}5{,}2{,}10{,}4{,}14{,}11{,}15{,}7{,}13{,}6{,}3{,}8)],$\\
$[(0{,}13{,}3{,}4{,}11{,}9{,}7{,}14{,}10{,}6{,}5{,}1{,}12{,}8{,}15{,}2)],[(0{,}14{,}3{,}7{,}11{,}10{,}8{,}13{,}2{,}12{,}5{,}9{,}6{,}4{,}15{,}1)],$\\
$[(0{,}15{,}5{,}4{,}2{,}11{,}3{,}14{,}1{,}13{,}7{,}9{,}8{,}12{,}6{,}10)]$
\item HWP$^*(16;4^{4},16^{11})${,}\\
$[(0{,}1{,}2{,}3){,}(4{,}5{,}6{,}7){,}(8{,}9{,}10{,}11){,}(12{,}13{,}14{,}15)],[(0{,}2{,}4{,}6){,}(8{,}10{,}12{,}14){,}(1{,}3{,}5{,}7){,}(9{,}11{,}13{,}15)],$\\
$[(0{,}3{,}9{,}12){,}(1{,}5{,}13{,}8){,}(2{,}6{,}14{,}7){,}(10{,}15{,}11{,}4)],[(0{,}4{,}15{,}2){,}(1{,}11{,}7{,}5){,}(3{,}14{,}9{,}13){,}(10{,}8{,}12{,}6)],$\\
$[(0{,}5{,}12{,}4{,}1{,}15{,}3{,}11{,}2{,}14{,}6{,}8{,}7{,}10{,}13{,}9)],[(0{,}6{,}11{,}5{,}10{,}7{,}12{,}9{,}2{,}15{,}4{,}8{,}3{,}13{,}1{,}14)],$\\
$[(0{,}7{,}3{,}2{,}1{,}4{,}9{,}5{,}8{,}6{,}12{,}11{,}15{,}10{,}14{,}13)],[(0{,}8{,}2{,}7{,}6{,}1{,}9{,}3{,}4{,}11{,}12{,}15{,}13{,}5{,}14{,}10)],$\\
$[(0{,}9{,}15{,}8{,}11{,}14{,}12{,}10{,}2{,}5{,}3{,}1{,}6{,}13{,}4{,}7)],[(0{,}10{,}1{,}7{,}8{,}4{,}2{,}13{,}6{,}9{,}14{,}11{,}3{,}12{,}5{,}15)],$\\
$[(0{,}11{,}1{,}8{,}5{,}2{,}9{,}7{,}13{,}12{,}3{,}10{,}6{,}15{,}14{,}4)],[(0{,}12{,}1{,}10{,}3{,}6{,}2{,}8{,}15{,}7{,}14{,}5{,}9{,}4{,}13{,}11)],$\\
$[(0{,}13{,}7{,}9{,}8{,}14{,}1{,}12{,}2{,}11{,}10{,}4{,}3{,}15{,}6{,}5)],[(0{,}14{,}3{,}7{,}15{,}1{,}13{,}2{,}10{,}5{,}11{,}9{,}6{,}4{,}12{,}8)],$\\
$[(0{,}15{,}5{,}4{,}14{,}2{,}12{,}7{,}11{,}6{,}3{,}8{,}13{,}10{,}9{,}1)]$
\item HWP$^*(16;4^{5},16^{10})${,}\\
$[(0{,}1{,}2{,}3){,}(4{,}5{,}6{,}7){,}(8{,}9{,}10{,}11){,}(12{,}13{,}14{,}15)],[(0{,}2{,}4{,}6){,}(8{,}10{,}12{,}14){,}(1{,}3{,}5{,}7){,}(9{,}11{,}13{,}15)],$\\
$[(0{,}3{,}9{,}12){,}(1{,}5{,}13{,}8){,}(2{,}6{,}14{,}7){,}(10{,}15{,}11{,}4)],[(0{,}4{,}15{,}2){,}(1{,}11{,}7{,}5){,}(3{,}14{,}9{,}13){,}(10{,}8{,}12{,}6)],$\\
$[(0{,}5{,}12{,}4){,}(1{,}15{,}3{,}11){,}(8{,}14{,}6{,}2){,}(7{,}10{,}13{,}9)],[(0{,}6{,}11{,}5{,}10{,}7{,}12{,}9{,}2{,}15{,}4{,}8{,}3{,}13{,}1{,}14)],$\\
$[(0{,}7{,}3{,}1{,}4{,}2{,}5{,}8{,}6{,}9{,}14{,}11{,}12{,}15{,}13{,}10)],[(0{,}8{,}2{,}1{,}6{,}3{,}4{,}7{,}9{,}5{,}15{,}10{,}14{,}13{,}12{,}11)],$\\
$[(0{,}9{,}1{,}7{,}6{,}4{,}3{,}2{,}10{,}5{,}11{,}15{,}14{,}12{,}8{,}13)],[(0{,}10{,}1{,}8{,}4{,}9{,}3{,}6{,}13{,}11{,}14{,}2{,}12{,}7{,}15{,}5)],$\\
$[(0{,}11{,}2{,}7{,}8{,}5{,}3{,}12{,}1{,}13{,}4{,}14{,}10{,}9{,}6{,}15)],[(0{,}12{,}2{,}9{,}4{,}11{,}3{,}10{,}6{,}8{,}15{,}7{,}13{,}5{,}14{,}1)],$\\
$[(0{,}13{,}2{,}11{,}6{,}1{,}10{,}3{,}15{,}8{,}7{,}14{,}4{,}12{,}5{,}9)],[(0{,}14{,}5{,}4{,}1{,}9{,}15{,}6{,}12{,}3{,}8{,}11{,}10{,}2{,}13{,}7)],$\\
$[(0{,}15{,}1{,}12{,}10{,}4{,}13{,}6{,}5{,}2{,}14{,}3{,}7{,}11{,}9{,}8)]$
\item HWP$^*(16;4^{6},16^{9})${,}\\
$
[(0{,}1{,}2{,}3){,}(4{,}5{,}6{,}7){,}(8{,}9{,}10{,}11){,}(12{,}13{,}14{,}15)],[(0{,}2{,}4{,}6){,}(8{,}10{,}12{,}14){,}(1{,}3{,}5{,}7){,}(9{,}11{,}13{,}15)],$\\
$[(0{,}3{,}9{,}12){,}(1{,}5{,}13{,}8){,}(2{,}6{,}14{,}7){,}(10{,}15{,}11{,}4)],[(0{,}4{,}15{,}2){,}(1{,}11{,}7{,}5){,}(3{,}14{,}9{,}13){,}(10{,}8{,}12{,}6)],$\\
$[(0{,}5{,}12{,}4){,}(1{,}15{,}3{,}11){,}(8{,}14{,}6{,}2){,}(7{,}10{,}13{,}9)],[(0{,}6{,}11{,}5){,}(1{,}7{,}12{,}9){,}(2{,}15{,}4{,}8){,}(3{,}13{,}10{,}14)],$\\
$[(0{,}7{,}3{,}1{,}4{,}2{,}5{,}8{,}6{,}9{,}14{,}11{,}15{,}13{,}12{,}10)],[(0{,}8{,}3{,}2{,}1{,}6{,}4{,}7{,}11{,}12{,}5{,}10{,}9{,}15{,}14{,}13)],$\\
$[(0{,}9{,}2{,}7{,}6{,}1{,}8{,}4{,}13{,}5{,}11{,}14{,}10{,}3{,}12{,}15)],[(0{,}10{,}1{,}9{,}3{,}4{,}11{,}2{,}12{,}7{,}8{,}13{,}6{,}15{,}5{,}14)],$\\
$[(0{,}11{,}3{,}6{,}5{,}2{,}9{,}4{,}14{,}12{,}8{,}15{,}10{,}7{,}13{,}1)],[(0{,}12{,}1{,}10{,}2{,}11{,}6{,}13{,}4{,}3{,}15{,}7{,}14{,}5{,}9{,}8)],$\\
$[(0{,}13{,}2{,}14{,}4{,}1{,}12{,}11{,}9{,}6{,}3{,}10{,}5{,}15{,}8{,}7)],[(0{,}14{,}1{,}13{,}7{,}15{,}6{,}12{,}2{,}10{,}4{,}9{,}5{,}3{,}8{,}11)],$\\
$[(0{,}15{,}1{,}14{,}2{,}13{,}11{,}10{,}6{,}8{,}5{,}4{,}12{,}3{,}7{,}9)]$
\item HWP$^*(16;4^{7},16^{8})${,}\\
$[(0{,}1{,}2{,}3){,}(4{,}5{,}6{,}7){,}(8{,}9{,}10{,}11){,}(12{,}13{,}14{,}15)],[(0{,}2{,}4{,}6){,}(8{,}10{,}12{,}14){,}(1{,}3{,}5{,}7){,}(9{,}11{,}13{,}15)],$\\
$[(0{,}3{,}9{,}12){,}(1{,}5{,}13{,}8){,}(2{,}6{,}14{,}7){,}(10{,}15{,}11{,}4)],[(0{,}4{,}15{,}2){,}(1{,}11{,}7{,}5){,}(3{,}14{,}9{,}13){,}(10{,}8{,}12{,}6)],$\\
$[(0{,}5{,}12{,}4){,}(1{,}15{,}3{,}11){,}(8{,}14{,}6{,}2){,}(7{,}10{,}13{,}9)],[(0{,}6{,}11{,}5){,}(1{,}7{,}12{,}9){,}(2{,}15{,}4{,}8){,}(3{,}13{,}10{,}14)],$\\
$[(0{,}7{,}15{,}10){,}(1{,}12{,}2{,}14){,}(3{,}6{,}13{,}11){,}(4{,}9{,}8{,}5)],[(0{,}8{,}3{,}1{,}4{,}2{,}5{,}9{,}6{,}12{,}10{,}7{,}11{,}15{,}14{,}13)],$\\
$[(0{,}9{,}2{,}1{,}6{,}3{,}4{,}7{,}8{,}13{,}12{,}11{,}14{,}10{,}5{,}15)],[(0{,}11{,}2{,}9{,}3{,}7{,}13{,}1{,}10{,}4{,}12{,}8{,}6{,}15{,}5{,}14)],$\\ 
$[(0{,}12{,}1{,}9{,}4{,}11{,}10{,}2{,}13{,}7{,}14{,}5{,}3{,}15{,}6{,}8)],[(0{,}10{,}1{,}8{,}4{,}3{,}2{,}7{,}6{,}9{,}14{,}12{,}15{,}13{,}5{,}11)],$\\
$[(0{,}13{,}4{,}14{,}2{,}11{,}9{,}15{,}8{,}7{,}3{,}12{,}5{,}10{,}6{,}1)],[(0{,}14{,}11{,}6{,}4{,}1{,}13{,}2{,}12{,}3{,}10{,}9{,}5{,}8{,}15{,}7)],$\\
$[(0{,}15{,}1{,}14{,}4{,}13{,}6{,}5{,}2{,}10{,}3{,}8{,}11{,}12{,}7{,}9)]$
\item HWP$^*(16;4^{8},16^{7})${,}\\ 
$[(0{,}1{,}2{,}3){,}(4{,}5{,}6{,}7){,}(8{,}9{,}10{,}11){,}(12{,}13{,}14{,}15)],[(0{,}2{,}4{,}6){,}(8{,}10{,}12{,}14){,}(1{,}3{,}5{,}7){,}(9{,}11{,}13{,}15)],$\\
$[(0{,}3{,}9{,}12){,}(1{,}5{,}13{,}8){,}(2{,}6{,}14{,}7){,}(10{,}15{,}11{,}4)],[(0{,}4{,}15{,}2){,}(1{,}11{,}7{,}5){,}(3{,}14{,}9{,}13){,}(10{,}8{,}12{,}6)],$\\
$[(0{,}5{,}12{,}4){,}(1{,}15{,}3{,}11){,}(8{,}14{,}6{,}2){,}(7{,}10{,}13{,}9)],[(0{,}6{,}11{,}5){,}(1{,}7{,}12{,}9){,}(2{,}15{,}4{,}8){,}(3{,}13{,}10{,}14)],$\\
$[(0{,}7{,}15{,}10){,}(1{,}12{,}2{,}14){,}(3{,}6{,}13{,}11){,}(4{,}9{,}8{,}5)],[(0{,}8{,}15{,}7){,}(1{,}9{,}2{,}13){,}(3{,}10{,}4{,}12){,}(14{,}11{,}6{,}5)],$\\
$[(0{,}9{,}3{,}1{,}4{,}2{,}5{,}8{,}6{,}12{,}10{,}7{,}11{,}15{,}14{,}13)],[(0{,}10{,}1{,}6{,}3{,}2{,}7{,}8{,}4{,}11{,}12{,}15{,}13{,}5{,}9{,}14)],$\\
$[(0{,}11{,}2{,}1{,}8{,}3{,}4{,}13{,}7{,}14{,}12{,}5{,}10{,}6{,}9{,}15)],[(0{,}12{,}1{,}10{,}2{,}9{,}4{,}14{,}5{,}3{,}7{,}13{,}6{,}15{,}8{,}11)],$\\
$[(0{,}13{,}2{,}12{,}8{,}7{,}3{,}15{,}6{,}4{,}1{,}14{,}10{,}5{,}11{,}9)],[(0{,}14{,}2{,}11{,}10{,}9{,}5{,}15{,}1{,}13{,}4{,}3{,}12{,}7{,}6{,}8)],$\\
$[(0{,}15{,}5{,}2{,}10{,}3{,}8{,}13{,}12{,}11{,}14{,}4{,}7{,}9{,}6{,}1)]$
\item HWP$^*(16;4^{9},16^{6})${,}\\
$[(0{,}1{,}2{,}3){,}(4{,}5{,}6{,}7){,}(8{,}9{,}10{,}11){,}(12{,}13{,}14{,}15)],[(0{,}2{,}4{,}6){,}(8{,}10{,}12{,}14){,}(1{,}3{,}5{,}7){,}(9{,}11{,}13{,}15)],$\\
$[(0{,}3{,}9{,}12){,}(1{,}5{,}13{,}8){,}(2{,}6{,}14{,}7){,}(10{,}15{,}11{,}4)],[(0{,}4{,}15{,}2){,}(1{,}11{,}7{,}5){,}(3{,}14{,}9{,}13){,}(10{,}8{,}12{,}6)],$\\
$[(0{,}5{,}12{,}4){,}(1{,}15{,}3{,}11){,}(8{,}14{,}6{,}2){,}(7{,}10{,}13{,}9)],[(0{,}6{,}11{,}5){,}(1{,}7{,}12{,}9){,}(2{,}15{,}4{,}8){,}(3{,}13{,}10{,}14)],$\\
$[(0{,}7{,}15{,}10){,}(1{,}12{,}2{,}14){,}(3{,}6{,}13{,}11){,}(4{,}9{,}8{,}5)],[(0{,}8{,}15{,}7){,}(1{,}9{,}2{,}13){,}(3{,}10{,}4{,}12){,}(14{,}11{,}6{,}5)],$\\
$[(0{,}9{,}15{,}8){,}(1{,}6{,}12{,}10){,}(2{,}5{,}11{,}14){,}(7{,}13{,}4{,}3)],[(0{,}10{,}2{,}1{,}4{,}7{,}3{,}8{,}6{,}9{,}5{,}15{,}14{,}13{,}12{,}11)]{,}$\\
$[(0{,}11{,}2{,}7{,}6{,}1{,}8{,}3{,}4{,}13{,}5{,}10{,}9{,}14{,}12{,}15)],[(0{,}12{,}1{,}10{,}3{,}2{,}9{,}6{,}4{,}14{,}5{,}8{,}7{,}11{,}15{,}13)]{,}$\\
$[(0{,}13{,}6{,}3{,}15{,}1{,}14{,}10{,}7{,}8{,}4{,}2{,}11{,}12{,}5{,}9)],[(0{,}14{,}4{,}11{,}10{,}6{,}15{,}5{,}2{,}12{,}8{,}13{,}7{,}9{,}3{,}1)]{,}$\\
$[(0{,}15{,}6{,}8{,}11{,}9{,}4{,}1{,}13{,}2{,}10{,}5{,}3{,}12{,}7{,}14)]$
\item HWP$^*(16;4^{10},16^{5})${,}\\
$[(0{,}1{,}2{,}3){,}(4{,}5{,}6{,}7){,}(8{,}9{,}10{,}11){,}(12{,}13{,}14{,}15)],[(0{,}2{,}4{,}6){,}(8{,}10{,}12{,}14){,}(1{,}3{,}5{,}7){,}(9{,}11{,}13{,}15)],$\\
$[(0{,}3{,}9{,}12){,}(1{,}5{,}13{,}8){,}(2{,}6{,}14{,}7){,}(10{,}15{,}11{,}4)],[(0{,}4{,}15{,}2){,}(1{,}11{,}7{,}5){,}(3{,}14{,}9{,}13){,}(10{,}8{,}12{,}6)],$\\
$[(0{,}5{,}12{,}4){,}(1{,}15{,}3{,}11){,}(8{,}14{,}6{,}2){,}(7{,}10{,}13{,}9)],[(0{,}6{,}11{,}5){,}(1{,}7{,}12{,}9){,}(2{,}15{,}4{,}8){,}(3{,}13{,}10{,}14)],$\\
$[(0{,}7{,}15{,}10){,}(1{,}12{,}2{,}14){,}(3{,}6{,}13{,}11){,}(4{,}9{,}8{,}5)],[(0{,}8{,}15{,}7){,}(1{,}9{,}2{,}13){,}(3{,}10{,}4{,}12){,}(14{,}11{,}6{,}5)],$\\
$[(0{,}9{,}15{,}8){,}(1{,}6{,}12{,}10){,}(2{,}5{,}11{,}14){,}(7{,}13{,}4{,}3)],[(0{,}10{,}5{,}9){,}(2{,}12{,}7{,}11){,}(3{,}15{,}1{,}8){,}(4{,}14{,}13{,}6)]{,}$\\
$[(0{,}11{,}9{,}3{,}1{,}4{,}2{,}7{,}8{,}13{,}12{,}5{,}10{,}6{,}15{,}14)],[(0{,}12{,}1{,}10{,}2{,}9{,}6{,}3{,}4{,}7{,}14{,}5{,}8{,}11{,}15{,}13)]{,}$\\
$[(0{,}13{,}2{,}11{,}10{,}7{,}6{,}9{,}14{,}12{,}15{,}5{,}3{,}8{,}4{,}1)],[(0{,}14{,}4{,}11{,}12{,}8{,}6{,}1{,}13{,}7{,}3{,}2{,}10{,}9{,}5{,}15)]{,}$\\
$[(0{,}15{,}6{,}8{,}7{,}9{,}4{,}13{,}5{,}2{,}1{,}14{,}10{,}3{,}12{,}11)]$
\item HWP$^*(16;4^{11},16^{4})${,}\\
$[(0{,}1{,}2{,}3){,}(4{,}5{,}6{,}7){,}(8{,}9{,}10{,}11){,}(12{,}13{,}14{,}15)],[(0{,}2{,}1{,}5){,}(3{,}6{,}4{,}8){,}(7{,}9{,}12{,}11){,}(10{,}14{,}13{,}15)],$\\
$[(0{,}3{,}5{,}2){,}(1{,}6{,}8{,}7){,}(4{,}12{,}9{,}13){,}(10{,}15{,}11{,}14)],[(0{,}4{,}3{,}7){,}(1{,}8{,}2{,}15){,}(5{,}9{,}11{,}13){,}(6{,}10{,}12{,}14)],$\\
$[(0{,}5{,}4{,}9){,}(1{,}12{,}2{,}8){,}(3{,}10{,}13{,}11){,}(6{,}14{,}7{,}15)],[(0{,}6{,}12{,}10){,}(1{,}3{,}11{,}5){,}(2{,}14{,}9{,}7){,}(4{,}13{,}8{,}15)],$\\
$[(0{,}7{,}12{,}4){,}(1{,}15{,}3{,}9){,}(2{,}13{,}6{,}11){,}(5{,}14{,}8{,}10)],[(0{,}8{,}14{,}11){,}(1{,}9{,}4{,}10){,}(2{,}5{,}12{,}6){,}(3{,}15{,}7{,}13)],$\\
$[(0{,}9{,}15{,}8){,}(1{,}7{,}5{,}11){,}(2{,}12{,}3{,}14){,}(4{,}6{,}13{,}10)],[(0{,}10{,}8{,}12){,}(1{,}14{,}3{,}13){,}(2{,}4{,}15{,}9){,}(5{,}7{,}11{,}6)]{,}$\\
$[(0{,}15{,}2{,}6){,}(1{,}11{,}4{,}14){,}(3{,}12{,}7{,}10){,}(5{,}13{,}9{,}8)],[(0{,}11{,}9{,}3{,}1{,}4{,}2{,}7{,}8{,}13{,}12{,}5{,}10{,}6{,}15{,}14)]{,}$\\
$[(0{,}12{,}1{,}10{,}2{,}9{,}6{,}3{,}4{,}7{,}14{,}5{,}8{,}11{,}15{,}13)],[(0{,}13{,}2{,}11{,}10{,}7{,}6{,}9{,}14{,}12{,}15{,}5{,}3{,}8{,}4{,}1)]{,}$\\
$[(0{,}14{,}4{,}11{,}12{,}8{,}6{,}1{,}13{,}7{,}3{,}2{,}10{,}9{,}5{,}15)]$
\item HWP$^*(16;4^{12},16^{3})${,}\\
$[(0{,}1{,}2{,}3){,}(4{,}5{,}6{,}7){,}(8{,}9{,}10{,}11){,}(12{,}13{,}14{,}15)],[(0{,}2{,}1{,}5){,}(3{,}6{,}4{,}8){,}(7{,}9{,}11{,}12){,}(10{,}14{,}13{,}15)],$\\
$[(0{,}3{,}2{,}4){,}(1{,}6{,}5{,}7){,}(8{,}10{,}12{,}14){,}(9{,}15{,}11{,}13)],[(0{,}4{,}3{,}7){,}(1{,}8{,}2{,}5){,}(6{,}12{,}10{,}15){,}(9{,}13{,}11{,}14)],$\\
$[(0{,}5{,}2{,}6{,}){,}(1{,}3{,}10{,}13){,}(4{,}9{,}12{,}11){,}(7{,}15{,}8{,}14)],[(0{,}6{,}1{,}9){,}(2{,}10{,}5{,}11){,}(3{,}13{,}4{,}14){,}(7{,}12{,}8{,}15)],$\\
$[(0{,}7{,}2{,}12){,}(1{,}11{,}6{,}10){,}(3{,}14{,}4{,}13){,}(5{,}15{,}9{,}8)],[(0{,}8{,}6{,}2){,}(1{,}13{,}10{,}9){,}(3{,}12{,}4{,}15){,}(5{,}14{,}11{,}7)],$\\
$[(0{,}9{,}2{,}8){,}(1{,}12{,}3{,}15){,}(4{,}6{,}14{,}10){,}(5{,}13{,}7{,}11)],[(0{,}10{,}3{,}11){,}(1{,}15{,}2{,}14){,}(4{,}12{,}9{,}5){,}(6{,}8{,}7{,}13)]{,}$\\
$[(0{,}14{,}2{,}15){,}(1{,}7{,}10{,}8){,}(3{,}9{,}4{,}11){,}(5{,}12{,}6{,}13)],[(0{,}15{,}4{,}10){,}(1{,}14{,}6{,}11){,}(2{,}13{,}8{,}12){,}(3{,}5{,}9{,}7)],$\\
$[(0{,}11{,}9{,}3{,}1{,}4{,}2{,}7{,}8{,}13{,}12{,}5{,}10{,}6{,}15{,}14)],[(0{,}12{,}1{,}10{,}2{,}9{,}6{,}3{,}4{,}7{,}14{,}5{,}8{,}11{,}15{,}13)],$\\
$[(0{,}13{,}2{,}11{,}10{,}7{,}6{,}9{,}14{,}12{,}15{,}5{,}3{,}8{,}4{,}1)]$
\item HWP$^*(16;4^{13},16^{2})${,}\\
$[(0{,}1{,}2{,}3){,}(4{,}5{,}6{,}7){,}(8{,}9{,}10{,}11){,}(12{,}13{,}14{,}15)],[(0{,}2{,}1{,}5){,}(3{,}6{,}4{,}8){,}(7{,}9{,}11{,}12){,}(10{,}14{,}13{,}15)],$\\
$[(0{,}3{,}2{,}4){,}(1{,}6{,}5{,}7){,}(8{,}10{,}9{,}13){,}(11{,}14{,}12{,}15)],[(0{,}4{,}1{,}7){,}(2{,}5{,}3{,}8){,}(6{,}12{,}10{,}15){,}(9{,}14{,}11{,}13)],$\\
$[(0{,}5{,}1{,}8){,}(2{,}6{,}9{,}12){,}(3{,}10{,}13{,}11){,}(4{,}14{,}7{,}15)],[(0{,}6{,}1{,}9){,}(2{,}8{,}4{,}11){,}(3{,}12{,}14{,}10){,}(5{,}15{,}7{,}13)],$\\
$[(0{,}7{,}2{,}10){,}(1{,}12{,}9{,}8){,}(3{,}15{,}5{,}14){,}(4{,}13{,}6{,}11)],[(0{,}8{,}6{,}2){,}(1{,}14{,}9{,}15){,}(3{,}5{,}13{,}7){,}(4{,}12{,}11{,}10)],$\\
$[(0{,}9{,}4{,}15){,}(1{,}13{,}2{,}14){,}(3{,}7{,}10{,}12){,}(5{,}11{,}6{,}8)],[(0{,}10{,}5{,}12){,}(1{,}15{,}3{,}11){,}(2{,}13{,}4{,}9){,}(6{,}14{,}8{,}7)]{,}$\\
$[(0{,}13{,}1{,}11){,}(2{,}15{,}8{,}14){,}(3{,}9{,}5{,}4){,}(6{,}10{,}7{,}12)],[(0{,}14{,}4{,}6){,}(1{,}3{,}13{,}10){,}(2{,}12{,}8{,}15){,}(5{,}9{,}7{,}11)],$\\
$[(0{,}15{,}9{,}1){,}(2{,}11{,}7{,}5){,}(3{,}14{,}6{,}13){,}(4{,}10{,}8{,}12)],[(0{,}11{,}9{,}3{,}1{,}4{,}2{,}7{,}8{,}13{,}12{,}5{,}10{,}6{,}15{,}14)],$\\
$[(0{,}12{,}1{,}10{,}2{,}9{,}6{,}3{,}4{,}7{,}14{,}5{,}8{,}11{,}15{,}13)]$
\item HWP$^*(16;4^{14},16^{1})${,}\\
$[(0{,}1{,}2{,}3){,}(4{,}5{,}6{,}7){,}(8{,}9{,}10{,}11){,}(12{,}13{,}14{,}15)],[(0{,}2{,}1{,}5){,}(3{,}4{,}6{,}8){,}(7{,}9{,}11{,}12){,}(10{,}14{,}13{,}15)],$\\
$[(0{,}3{,}2{,}4){,}(1{,}6{,}5{,}7){,}(8{,}10{,}9{,}13){,}(11{,}14{,}12{,}15)],[(0{,}4{,}1{,}7){,}(2{,}5{,}3{,}6){,}(8{,}12{,}9{,}14){,}(10{,}15{,}13{,}11)],$\\
$[(0{,}5{,}1{,}8){,}(2{,}6{,}3{,}7){,}(4{,}14{,}9{,}15){,}(10{,}12{,}11{,}13)],[(0{,}6{,}1{,}9){,}(2{,}8{,}4{,}10){,}(3{,}12{,}14{,}11){,}(5{,}13{,}7{,}15)],$\\
$[(0{,}7{,}3{,}10){,}(1{,}11{,}2{,}9){,}(4{,}12{,}6{,}13){,}(5{,}15{,}8{,}14)],[(0{,}8{,}1{,}12){,}(2{,}10{,}5{,}14){,}(3{,}9{,}7{,}13){,}(4{,}11{,}15{,}6)],$\\
$[(0{,}9{,}12{,}2){,}(1{,}3{,}8{,}15){,}(4{,}13{,}5{,}11){,}(6{,}14{,}10{,}7)],[(0{,}10{,}13{,}1){,}(2{,}15{,}9{,}8){,}(3{,}14{,}6{,}12){,}(4{,}7{,}11{,}5)]{,}$\\
$[(0{,}12{,}8{,}6){,}(1{,}10{,}3{,}11){,}(2{,}13{,}9{,}5){,}(4{,}15{,}7{,}14)],[(0{,}13{,}6{,}11){,}(1{,}15{,}2{,}14){,}(3{,}5{,}9{,}4){,}(7{,}12{,}10{,}8)],$\\
$[(0{,}14{,}3{,}15){,}(1{,}13{,}2{,}12){,}(4{,}9{,}6{,}10){,}(5{,}8{,}11{,}7)],[(0{,}15{,}3{,}13){,}(1{,}14{,}7{,}10){,}(2{,}11{,}6{,}9){,}(4{,}8{,}5{,}12)],$\\
$[(0{,}11{,}9{,}3{,}1{,}4{,}2{,}7{,}8{,}13{,}12{,}5{,}10{,}6{,}15{,}14)]$
\item HWP$^*(15;3^{1},15^{13})${,}\\
$[(0{,}1{,}2){,}(3{,}4{,}5){,}(6{,}7{,}8){,}(9{,}10{,}11){,}(12{,}13{,}14)],[(0{,}2{,}4{,}8{,}10{,}6{,}12{,}14{,}1{,}3{,}5{,}7{,}9{,}11{,}13)]{,}$\\
$[(0{,}3{,}9{,}1{,}8{,}4{,}6{,}14{,}7{,}5{,}10{,}13{,}2{,}12{,}11)],[(0{,}4{,}7{,}1{,}9{,}3{,}2{,}14{,}13{,}6{,}5{,}8{,}11{,}12{,}10)]{,}$\\
$[(0{,}5{,}13{,}1{,}4{,}10{,}2{,}9{,}14{,}3{,}6{,}11{,}7{,}12{,}8)],[(0{,}6{,}1{,}5{,}2{,}3{,}7{,}4{,}9{,}8{,}13{,}11{,}14{,}10{,}12)]{,}$\\
$[(0{,}7{,}2{,}1{,}6{,}3{,}8{,}5{,}9{,}13{,}12{,}4{,}11{,}10{,}14)],[(0{,}8{,}1{,}7{,}3{,}10{,}4{,}2{,}13{,}9{,}12{,}5{,}14{,}11{,}6)]{,}$\\
$[(0{,}9{,}2{,}5{,}1{,}10{,}3{,}11{,}4{,}14{,}6{,}13{,}8{,}12{,}7)],[(0{,}10{,}1{,}11{,}2{,}6{,}4{,}3{,}14{,}8{,}7{,}13{,}5{,}12{,}9)]{,}$\\
$[(0{,}11{,}1{,}12{,}2{,}7{,}6{,}10{,}8{,}14{,}9{,}5{,}4{,}13{,}3)],[(0{,}12{,}1{,}14{,}2{,}10{,}5{,}6{,}9{,}7{,}11{,}8{,}3{,}13{,}4)]{,}$\\
$[(0{,}13{,}10{,}7{,}14{,}5{,}11{,}3{,}12{,}6{,}2{,}8{,}9{,}4{,}1)],[(0{,}14{,}4{,}12{,}3{,}1{,}13{,}7{,}10{,}9{,}6{,}8{,}2{,}11{,}5)]$
\item HWP$^*(15;3^{3},15^{11})${,}\\
$[(0{,}1{,}2){,}(3{,}4{,}5){,}(6{,}7{,}8){,}(9{,}10{,}11){,}(12{,}13{,}14)],[(0{,}2{,}4){,}(6{,}8{,}10){,}(12{,}14{,}1){,}(3{,}5{,}7){,}(9{,}11{,}13)],$\\
$[(0{,}3{,}9){,}(1{,}8{,}2){,}(4{,}14{,}7){,}(5{,}10{,}13){,}(6{,}12{,}11)],[(0{,}4{,}7{,}1{,}9{,}3{,}2{,}14{,}13{,}6{,}5{,}8{,}11{,}12{,}10)],$\\
$[(0{,}5{,}13{,}1{,}4{,}10{,}2{,}9{,}14{,}3{,}6{,}11{,}7{,}12{,}8)],[(0{,}6{,}1{,}3{,}7{,}2{,}5{,}4{,}8{,}9{,}13{,}11{,}14{,}10{,}12)],$\\
$[(0{,}7{,}5{,}1{,}6{,}2{,}3{,}8{,}12{,}9{,}4{,}13{,}10{,}14{,}11)],[(0{,}8{,}1{,}5{,}2{,}6{,}3{,}13{,}4{,}11{,}10{,}9{,}12{,}7{,}14)],$\\
$[(0{,}9{,}1{,}7{,}6{,}4{,}12{,}2{,}10{,}3{,}11{,}5{,}14{,}8{,}13)],[(0{,}10{,}1{,}11{,}2{,}7{,}9{,}5{,}12{,}6{,}13{,}8{,}14{,}4{,}3)],$\\
$[(0{,}11{,}1{,}10{,}4{,}2{,}8{,}3{,}14{,}6{,}9{,}7{,}13{,}12{,}5)],[(0{,}12{,}4{,}9{,}6{,}14{,}2{,}13{,}7{,}10{,}8{,}5{,}11{,}3{,}1)],$\\
$[(0{,}13{,}3{,}12{,}1{,}14{,}5{,}9{,}2{,}11{,}8{,}4{,}6{,}10{,}7)],[(0{,}14{,}9{,}8{,}7{,}11{,}4{,}1{,}13{,}2{,}12{,}3{,}10{,}5{,}6)]$
\item HWP$^*(15;3^{5},15^{9})${,}\\
$[(0{,}1{,}2){,}(3{,}4{,}5){,}(6{,}7{,}8){,}(9{,}10{,}11){,}(12{,}13{,}14)],[(0{,}2{,}4){,}(6{,}8{,}10){,}(12{,}14{,}1){,}(3{,}5{,}7){,}(9{,}11{,}13)],$\\
$[(0{,}3{,}9){,}(1{,}8{,}2){,}(4{,}14{,}7){,}(5{,}10{,}13){,}(6{,}12{,}11)],[(0{,}4{,}7){,}(1{,}9{,}3){,}(2{,}14{,}13){,}(5{,}11{,}8){,}(10{,}12{,}6)],$\\
$[(0{,}5{,}13){,}(1{,}4{,}10){,}(2{,}9{,}14){,}(3{,}6{,}11){,}(7{,}12{,}8)],[(0{,}6{,}1{,}3{,}2{,}5{,}4{,}8{,}9{,}7{,}13{,}10{,}14{,}11{,}12)],$\\
$[(0{,}7{,}1{,}5{,}2{,}3{,}8{,}4{,}6{,}13{,}11{,}14{,}9{,}12{,}10)],[(0{,}8{,}1{,}6{,}2{,}7{,}5{,}9{,}13{,}12{,}3{,}14{,}10{,}4{,}11)],$\\
$[(0{,}9{,}1{,}7{,}2{,}6{,}3{,}11{,}10{,}5{,}12{,}4{,}13{,}8{,}14)],[(0{,}10{,}2{,}8{,}3{,}7{,}11{,}1{,}13{,}4{,}12{,}9{,}5{,}14{,}6)],$\\
$[(0{,}11{,}2{,}10{,}3{,}12{,}1{,}14{,}4{,}9{,}8{,}13{,}7{,}6{,}5)],[(0{,}12{,}2{,}11{,}4{,}3{,}13{,}1{,}10{,}7{,}9{,}6{,}14{,}5{,}8)],$\\
$[(0{,}13{,}6{,}9{,}4{,}2{,}12{,}7{,}14{,}3{,}10{,}8{,}11{,}5{,}1)],[(0{,}14{,}8{,}12{,}5{,}6{,}4{,}1{,}11{,}7{,}10{,}9{,}2{,}13{,}3)]$
\item HWP$^*(15;3^{7},15^{7})${,}\\
$[(0{,}1{,}2){,}(3{,}4{,}5){,}(6{,}7{,}8){,}(9{,}10{,}11){,}(12{,}13{,}14)],[(0{,}2{,}4){,}(6{,}8{,}10){,}(12{,}14{,}1){,}(3{,}5{,}7){,}(9{,}11{,}13)],$\\
$[(0{,}3{,}9){,}(1{,}8{,}2){,}(4{,}14{,}7){,}(5{,}10{,}13){,}(6{,}12{,}11)],[(0{,}4{,}7){,}(1{,}9{,}3){,}(2{,}14{,}13){,}(5{,}11{,}8){,}(10{,}12{,}6)],$\\
$[(0{,}5{,}13){,}(1{,}4{,}10){,}(2{,}9{,}14){,}(3{,}6{,}11){,}(7{,}12{,}8)],[(0{,}6{,}1){,}(2{,}13{,}10){,}(3{,}8{,}14){,}(4{,}11{,}5){,}(7{,}9{,}12)],$\\
$[(0{,}7{,}11){,}(2{,}6{,}5){,}(1{,}14{,}9){,}(4{,}12{,}10){,}(3{,}13{,}8)],[(0{,}8{,}1{,}3{,}2{,}5{,}6{,}4{,}9{,}13{,}7{,}10{,}14{,}11{,}12)]{,}$\\
$[(0{,}9{,}2{,}3{,}7{,}1{,}5{,}8{,}12{,}4{,}6{,}13{,}11{,}14{,}10)],[(0{,}10{,}3{,}11{,}1{,}6{,}2{,}7{,}5{,}14{,}4{,}13{,}12{,}9{,}8)]{,}$\\
$[(0{,}11{,}2{,}8{,}4{,}1{,}10{,}7{,}13{,}3{,}12{,}5{,}9{,}6{,}14)],[(0{,}12{,}2{,}11{,}4{,}3{,}10{,}8{,}9{,}7{,}14{,}5{,}1{,}13{,}6)]{,}$\\
$[(0{,}13{,}1{,}7{,}2{,}12{,}3{,}14{,}6{,}9{,}4{,}8{,}11{,}10{,}5)],[(0{,}14{,}8{,}13{,}4{,}2{,}10{,}9{,}5{,}12{,}1{,}11{,}7{,}6{,}3)]$
\item HWP$^*(15;3^{9},15^{5})${,}\\
$[(0{,}1{,}2){,}(3{,}4{,}5){,}(6{,}7{,}8){,}(9{,}10{,}11){,}(12{,}13{,}14)],[(0{,}2{,}4){,}(6{,}8{,}10){,}(12{,}14{,}1){,}(3{,}5{,}7){,}(9{,}11{,}13)],$\\
$[(0{,}3{,}9){,}(1{,}8{,}2){,}(4{,}14{,}7){,}(5{,}10{,}13){,}(6{,}12{,}11)],[(0{,}4{,}7){,}(1{,}9{,}3){,}(2{,}14{,}13){,}(5{,}11{,}8){,}(10{,}12{,}6)],$\\
$[(0{,}5{,}13){,}(1{,}4{,}10){,}(2{,}9{,}14){,}(3{,}6{,}11){,}(7{,}12{,}8)],[(0{,}6{,}1){,}(2{,}13{,}10){,}(3{,}8{,}14){,}(4{,}11{,}5){,}(7{,}9{,}12)],$\\
$[(0{,}7{,}11){,}(2{,}6{,}5){,}(1{,}14{,}9){,}(4{,}12{,}10){,}(3{,}13{,}8)],[(0{,}8{,}12){,}(1{,}13{,}4){,}(2{,}10{,}3){,}(7{,}14{,}11){,}(5{,}6{,}9)],$\\
$[(0{,}9{,}8){,}(2{,}5{,}12){,}(11{,}4{,}13){,}(3{,}14{,}6){,}(1{,}10{,}7)],[(0{,}10{,}5{,}1{,}3{,}7{,}2{,}8{,}11{,}12{,}4{,}9{,}13{,}6{,}14)],$\\
$[(0{,}11{,}1{,}5{,}8{,}9{,}2{,}7{,}10{,}14{,}4{,}6{,}13{,}12{,}3)],[(0{,}12{,}1{,}7{,}13{,}3{,}10{,}8{,}4{,}2{,}11{,}14{,}5{,}9{,}6)],$\\
$[(0{,}13{,}7{,}5{,}14{,}8{,}1{,}6{,}2{,}12{,}9{,}4{,}3{,}11{,}10)],[(0{,}14{,}10{,}9{,}7{,}6{,}4{,}8{,}13{,}1{,}11{,}2{,}3{,}12{,}5)]$
\item HWP$^*(15;3^{11},15^{3})${,}\\
$[(0{,}1{,}2){,}(3{,}4{,}5){,}(6{,}7{,}8){,}(9{,}10{,}11){,}(12{,}13{,}14)],[(0{,}2{,}4){,}(6{,}8{,}10){,}(12{,}14{,}1){,}(3{,}5{,}7){,}(9{,}11{,}13)],$\\
$[(0{,}3{,}9){,}(1{,}8{,}2){,}(4{,}14{,}7){,}(5{,}10{,}13){,}(6{,}12{,}11)],[(0{,}4{,}7){,}(1{,}9{,}3){,}(2{,}14{,}13){,}(5{,}11{,}8){,}(10{,}12{,}6)],$\\
$[(0{,}5{,}13){,}(1{,}4{,}10){,}(2{,}9{,}14){,}(3{,}6{,}11){,}(7{,}12{,}8)],[(0{,}6{,}1){,}(2{,}13{,}10){,}(3{,}8{,}14){,}(4{,}11{,}5){,}(7{,}9{,}12)],$\\
$[(0{,}7{,}11){,}(2{,}6{,}5){,}(1{,}14{,}9){,}(4{,}12{,}10){,}(3{,}13{,}8)],[(0{,}8{,}12){,}(1{,}13{,}4){,}(2{,}10{,}3){,}(7{,}14{,}11){,}(5{,}6{,}9)],$\\
$[(0{,}9{,}8){,}(2{,}5{,}12){,}(11{,}4{,}13){,}(3{,}14{,}6){,}(1{,}10{,}7)]{,}[(0{,}10{,}14){,}(12{,}4{,}3){,}(2{,}8{,}11){,}(7{,}5{,}9){,}(1{,}6{,}13)],$\\
$[(0{,}11{,}10){,}(1{,}5{,}8){,}(2{,}12{,}9){,}(3{,}7{,}13){,}(4{,}6{,}14)],[(0{,}12{,}1{,}3{,}10{,}9{,}6{,}4{,}8{,}13{,}7{,}2{,}11{,}14{,}5)],$\\
$[(0{,}13{,}12{,}5{,}14{,}10{,}8{,}9{,}4{,}2{,}3{,}11{,}1{,}7{,}6)],[(0{,}14{,}8{,}4{,}9{,}13{,}6{,}2{,}7{,}10{,}5{,}1{,}11{,}12{,}3)]$
\item HWP$^*(15;5^{1},15^{13})${,}\\
$[(0{,}1{,}2{,}3{,}4){,}(5{,}6{,}7{,}8{,}9){,}(10{,}11{,}12{,}13{,}14)],[(0{,}2{,}8{,}6{,}4{,}10{,}14{,}12{,}1{,}3{,}5{,}7{,}9{,}11{,}13)]{,}$\\
$[(0{,}3{,}9{,}1{,}8{,}2{,}4{,}14{,}7{,}5{,}10{,}13{,}6{,}12{,}11)],[(0{,}4{,}7{,}1{,}9{,}3{,}2{,}14{,}13{,}5{,}11{,}8{,}10{,}12{,}6)]{,}$\\
$[(0{,}5{,}13{,}1{,}4{,}2{,}10{,}9{,}14{,}3{,}6{,}11{,}7{,}12{,}8)],[(0{,}6{,}1{,}5{,}2{,}7{,}3{,}8{,}4{,}9{,}13{,}12{,}14{,}11{,}10)]{,}$\\
$[(0{,}7{,}2{,}1{,}6{,}3{,}10{,}4{,}5{,}8{,}13{,}11{,}14{,}9{,}12)],[(0{,}8{,}1{,}7{,}4{,}3{,}11{,}2{,}12{,}9{,}6{,}13{,}10{,}5{,}14)]{,}$\\
$[(0{,}9{,}2{,}5{,}1{,}10{,}3{,}13{,}4{,}11{,}6{,}14{,}8{,}12{,}7)],[(0{,}10{,}1{,}11{,}3{,}7{,}6{,}2{,}13{,}8{,}14{,}4{,}12{,}5{,}9)]{,}$\\
$[(0{,}11{,}1{,}12{,}2{,}6{,}9{,}4{,}13{,}7{,}10{,}8{,}3{,}14{,}5)],[(0{,}12{,}10{,}6{,}5{,}4{,}8{,}11{,}9{,}7{,}13{,}3{,}1{,}14{,}2)]{,}$\\
$[(0{,}13{,}9{,}8{,}7{,}14{,}6{,}10{,}2{,}11{,}5{,}3{,}12{,}4{,}1)],[(0{,}14{,}1{,}13{,}2{,}9{,}10{,}7{,}11{,}4{,}6{,}8{,}5{,}12{,}3)]$
\item HWP$^*(15;5^{3},15^{11})${,}\\
$[(0{,}1{,}2{,}3{,}4){,}(5{,}6{,}7{,}8{,}9){,}(10{,}11{,}12{,}13{,}14)],[(0{,}2{,}8{,}6{,}5){,}(10{,}14{,}12{,}1{,}3){,}(7{,}4{,}9{,}11{,}13)]{,}$\\
$[(0{,}3{,}9{,}1{,}7){,}(2{,}4{,}14{,}8{,}5){,}(10{,}13{,}6{,}12{,}11)],[(0{,}4{,}7{,}1{,}9{,}3{,}2{,}14{,}13{,}5{,}11{,}8{,}10{,}12{,}6)]{,}$\\
$[(0{,}5{,}13{,}1{,}4{,}2{,}10{,}9{,}14{,}3{,}6{,}11{,}7{,}12{,}8)],[(0{,}6{,}1{,}5{,}3{,}7{,}2{,}9{,}4{,}10{,}8{,}13{,}12{,}14{,}11)]{,}$\\
$[(0{,}7{,}3{,}1{,}6{,}2{,}5{,}4{,}11{,}14{,}9{,}13{,}8{,}12{,}10)],[(0{,}8{,}1{,}10{,}2{,}6{,}3{,}5{,}7{,}14{,}4{,}13{,}11{,}9{,}12)]{,}$\\
$[(0{,}9{,}2{,}1{,}8{,}3{,}11{,}4{,}5{,}12{,}7{,}13{,}10{,}6{,}14)],[(0{,}10{,}1{,}11{,}2{,}7{,}5{,}8{,}4{,}12{,}3{,}14{,}6{,}13{,}9)]{,}$\\
$[(0{,}11{,}1{,}12{,}2{,}13{,}4{,}8{,}14{,}5{,}9{,}7{,}6{,}10{,}3)],[(0{,}12{,}5{,}10{,}4{,}1{,}14{,}7{,}9{,}6{,}8{,}11{,}3{,}13{,}2)]{,}$\\
$[(0{,}13{,}3{,}8{,}2{,}12{,}4{,}6{,}9{,}10{,}7{,}11{,}5{,}14{,}1)],[(0{,}14{,}2{,}11{,}6{,}4{,}3{,}12{,}9{,}8{,}7{,}10{,}5{,}1{,}13)]$
\item HWP$^*(15;5^{5},15^{9})${,}\\
$[(0{,}1{,}2{,}3{,}4){,}(5{,}6{,}7{,}8{,}9){,}(10{,}11{,}12{,}13{,}14)],[(0{,}2{,}8{,}6{,}5){,}(10{,}14{,}12{,}1{,}3){,}(7{,}4{,}9{,}11{,}13)]{,}$\\
$[(0{,}3{,}9{,}1{,}7){,}(2{,}4{,}14{,}8{,}5){,}(10{,}13{,}6{,}12{,}11)],[(0{,}4{,}7{,}1{,}9){,}(3{,}2{,}14{,}13{,}5){,}(11{,}8{,}10{,}12{,}6)]{,}$\\
$[(0{,}5{,}13{,}2{,}1){,}(4{,}10{,}9{,}14{,}6){,}(3{,}11{,}7{,}12{,}8)],[(0{,}6{,}1{,}4{,}2{,}5{,}7{,}3{,}8{,}11{,}14{,}9{,}13{,}12{,}10)]{,}$\\
$[(0{,}7{,}2{,}6{,}3{,}1{,}5{,}4{,}12{,}14{,}11{,}9{,}10{,}8{,}13)],[(0{,}8{,}1{,}6{,}2{,}7{,}9{,}12{,}3{,}13{,}10{,}4{,}11{,}5{,}14)]{,}$\\
$[(0{,}9{,}2{,}10{,}1{,}8{,}4{,}3{,}6{,}14{,}5{,}12{,}7{,}13{,}11)],[(0{,}10{,}2{,}9{,}3{,}5{,}1{,}12{,}4{,}13{,}8{,}14{,}7{,}11{,}6)]{,}$\\
$[(0{,}11{,}1{,}10{,}3{,}14{,}2{,}12{,}5{,}9{,}7{,}6{,}13{,}4{,}8)],[(0{,}12{,}9{,}4{,}5{,}10{,}6{,}8{,}7{,}14{,}1{,}11{,}2{,}13{,}3)]{,}$\\
$[(0{,}13{,}1{,}14{,}3{,}7{,}10{,}5{,}11{,}4{,}6{,}9{,}8{,}12{,}2)],[(0{,}14{,}4{,}1{,}13{,}9{,}6{,}10{,}7{,}5{,}8{,}2{,}11{,}3{,}12)]$
\item HWP$^*(15;5^{7},15^{7})${,}\\
$[(0{,}1{,}2{,}3{,}4){,}(5{,}6{,}7{,}8{,}9){,}(10{,}11{,}12{,}13{,}14)],[(0{,}2{,}8{,}6{,}5){,}(10{,}14{,}12{,}1{,}3){,}(7{,}4{,}9{,}11{,}13)]{,}$\\
$[(0{,}3{,}9{,}1{,}7){,}(2{,}4{,}14{,}8{,}5){,}(10{,}13{,}6{,}12{,}11)],[(0{,}4{,}7{,}1{,}9){,}(3{,}2{,}14{,}13{,}5){,}(11{,}8{,}10{,}12{,}6)]{,}$\\
$[(0{,}5{,}13{,}2{,}1){,}(4{,}10{,}9{,}14{,}6){,}(3{,}11{,}7{,}12{,}8)],[(0{,}6{,}13{,}11{,}2){,}(1{,}10{,}3{,}7{,}14){,}(4{,}8{,}12{,}5{,}9)]{,}$\\
$[(0{,}7{,}6{,}2{,}11){,}(1{,}14{,}3{,}8{,}4){,}(13{,}10{,}5{,}12{,}9)],[(0{,}8{,}1{,}4{,}2{,}5{,}7{,}3{,}13{,}9{,}12{,}14{,}11{,}6{,}10)]{,}$\\
$[(0{,}9{,}2{,}6{,}1{,}5{,}4{,}3{,}12{,}10{,}7{,}13{,}8{,}11{,}14)],[(0{,}10{,}1{,}6{,}3{,}5{,}8{,}14{,}2{,}9{,}7{,}11{,}4{,}13{,}12)]{,}$\\
$[(0{,}11{,}1{,}8{,}7{,}5{,}10{,}2{,}13{,}4{,}12{,}3{,}14{,}9{,}6)],[(0{,}12{,}2{,}7{,}10{,}8{,}13{,}1{,}11{,}5{,}14{,}4{,}6{,}9{,}3)]{,}$\\
$[(0{,}13{,}3{,}6{,}14{,}5{,}1{,}12{,}7{,}2{,}10{,}4{,}11{,}9{,}8)],[(0{,}14{,}7{,}9{,}10{,}6{,}8{,}2{,}12{,}4{,}5{,}11{,}3{,}1{,}13)]{,}$
\item HWP$^*(15;5^{9},15^{5})${,}\\
$[(0{,}1{,}2{,}3{,}4){,}(5{,}6{,}7{,}8{,}9){,}(10{,}11{,}12{,}13{,}14)],[(0{,}2{,}8{,}6{,}5){,}(10{,}14{,}12{,}1{,}3){,}(7{,}4{,}9{,}11{,}13)]{,}$\\
$[(0{,}3{,}9{,}1{,}7){,}(2{,}4{,}14{,}8{,}5){,}(10{,}13{,}6{,}12{,}11)],[(0{,}4{,}7{,}1{,}9){,}(3{,}2{,}14{,}13{,}5){,}(11{,}8{,}10{,}12{,}6)]{,}$\\
$[(0{,}5{,}13{,}2{,}1){,}(4{,}10{,}9{,}14{,}6){,}(3{,}11{,}7{,}12{,}8)],[(0{,}6{,}13{,}11{,}2){,}(1{,}10{,}3{,}7{,}14){,}(4{,}8{,}12{,}5{,}9)]{,}$\\
$[(0{,}7{,}6{,}2{,}11){,}(1{,}14{,}3{,}8{,}4){,}(13{,}10{,}5{,}12{,}9)],[(0{,}8{,}13{,}3{,}12){,}(1{,}11{,}4{,}6{,}10){,}(2{,}5{,}14{,}7{,}9)]{,}$\\
$[(0{,}9{,}10{,}6{,}8){,}(1{,}13{,}12{,}3{,}5){,}(2{,}7{,}11{,}14{,}4)],[(0{,}10{,}2{,}6{,}1{,}4{,}3{,}13{,}8{,}7{,}5{,}11{,}9{,}12{,}14)]{,}$\\ 
$[(0{,}11{,}1{,}5{,}4{,}12{,}7{,}10{,}8{,}14{,}2{,}13{,}9{,}3{,}6)],[(0{,}12{,}4{,}13{,}1{,}8{,}2{,}9{,}6{,}14{,}11{,}5{,}10{,}7{,}3)]{,}$\\
$[(0{,}13{,}4{,}11{,}3{,}14{,}5{,}8{,}1{,}6{,}9{,}7{,}2{,}12{,}10)],[(0{,}14{,}9{,}8{,}11{,}6{,}3{,}1{,}12{,}2{,}10{,}4{,}5{,}7{,}13)]$
\item HWP$^*(15;5^{11},15^{3})${,}\\
$[(0{,}1{,}2{,}3{,}4){,}(5{,}6{,}7{,}8{,}9){,}(10{,}11{,}12{,}13{,}14)],[(0{,}2{,}8{,}6{,}5){,}(10{,}14{,}12{,}1{,}3){,}(7{,}4{,}9{,}11{,}13)]{,}$\\
$[(0{,}3{,}9{,}1{,}7){,}(2{,}4{,}14{,}8{,}5){,}(10{,}13{,}6{,}12{,}11)],[(0{,}4{,}7{,}1{,}9){,}(3{,}2{,}14{,}13{,}5){,}(11{,}8{,}10{,}12{,}6)]{,}$\\
$[(0{,}5{,}13{,}2{,}1){,}(4{,}10{,}9{,}14{,}6){,}(3{,}11{,}7{,}12{,}8)],[(0{,}6{,}2{,}7{,}3){,}(1{,}12{,}4{,}5{,}10){,}(8{,}13{,}11{,}14{,}9)]{,}$\\
$[(0{,}7{,}9{,}4{,}13){,}(1{,}8{,}12{,}3{,}5){,}(2{,}10{,}6{,}14{,}11)],[(0{,}8{,}4{,}2{,}11){,}(1{,}14{,}7{,}6{,}3){,}(5{,}12{,}9{,}13{,}10)]{,}$\\
$[(0{,}9{,}6{,}13{,}12){,}(1{,}10{,}7{,}11{,}4){,}(2{,}5{,}14{,}3{,}8],[(0{,}12{,}5{,}9{,}2){,}(1{,}13{,}3{,}7{,}14){,}(4{,}8{,}11{,}6{,}10)]{,}$\\
$[(0{,}14{,}4{,}6{,}8){,}(1{,}11{,}5{,}7{,}13){,}(2{,}9{,}10{,}3{,}12)],[(0{,}10{,}2{,}6{,}1{,}4{,}3{,}13{,}8{,}7{,}5{,}11{,}9{,}12{,}14)]{,}$\\
$[(0{,}11{,}1{,}5{,}4{,}12{,}7{,}10{,}8{,}14{,}2{,}13{,}9{,}3{,}6)],[(0{,}13{,}4{,}11{,}3{,}14{,}5{,}8{,}1{,}6{,}9{,}7{,}2{,}12{,}10)]$
\item HWP$^*(15;5^{13},15^{1})${,}\\
$[(0{,}1{,}2{,}3{,}4){,}(5{,}6{,}7{,}8{,}9){,}(10{,}11{,}12{,}13{,}14)],[(0{,}2{,}8{,}6{,}5){,}(10{,}14{,}12{,}1{,}3){,}(7{,}4{,}9{,}11{,}13)]{,}$\\
$[(0{,}3{,}9{,}1{,}7){,}(2{,}4{,}14{,}8{,}5){,}(10{,}13{,}6{,}12{,}11)],[(0{,}4{,}7{,}1{,}9){,}(3{,}2{,}14{,}13{,}5){,}(11{,}8{,}10{,}12{,}6)]{,}$\\
$[(0{,}5{,}13{,}2{,}1){,}(4{,}10{,}9{,}14{,}6){,}(3{,}11{,}7{,}12{,}8)],[(0{,}6{,}13{,}11{,}2){,}(1{,}10{,}3{,}7{,}14){,}(4{,}8{,}12{,}5{,}9)]{,}$\\
$[(0{,}7{,}2{,}11{,}14){,}(1{,}6{,}9{,}12{,}3){,}(4{,}13{,}10{,}5{,}8)],[(0{,}8{,}13{,}3{,}12){,}(1{,}11{,}4{,}6{,}10){,}(2{,}5{,}14{,}7{,}9)]{,}$\\
$[(0{,}9{,}6{,}2{,}10){,}(1{,}8{,}7{,}11{,}5){,}(3{,}13{,}12{,}14{,}4)],[(0{,}10{,}4{,}11{,}3){,}(1{,}14{,}5{,}7{,}6){,}(2{,}12{,}9{,}13{,}8)]{,}$\\
$[(0{,}12{,}4{,}1{,}13){,}(2{,}9{,}7{,}5{,}10){,}(3{,}8{,}11{,}6{,}14)],[(0{,}13{,}4{,}5{,}11){,}(1{,}12{,}2{,}6{,}8){,}(3{,}14{,}9{,}10{,}7)]{,}$\\
$[(0{,}14{,}11{,}9{,}8){,}(1{,}4{,}2{,}7{,}13){,}(3{,}5{,}12{,}10{,}6)],[(0{,}11{,}1{,}5{,}4{,}12{,}7{,}10{,}8{,}14{,}2{,}13{,}9{,}3{,}6)]$
\item HWP$^*(15;3^{1},5^{13})${,}\\
$[(0{,}1{,}2){,}(3{,}4{,}5){,}(6{,}7{,}8){,}(9{,}10{,}11){,}(12{,}13{,}14)],[(0{,}2{,}1{,}4{,}7){,}(3{,}14{,}5{,}9{,}8){,}(6{,}12{,}11{,}13{,}10)]{,}$\\
$[(0{,}3{,}9{,}1{,}8){,}(2{,}11{,}5{,}6{,}13){,}(4{,}12{,}14{,}7{,}10)],[(0{,}4{,}2{,}9{,}14){,}(1{,}6{,}10{,}7{,}12){,}(3{,}13{,}8{,}5{,}11)]{,}$\\
$[(0{,}5{,}14{,}13{,}1){,}(2{,}6{,}4{,}10{,}12){,}(3{,}7{,}11{,}8{,}9)],[(0{,}6{,}1{,}13{,}11){,}(2{,}10{,}3{,}12{,}7){,}(4{,}8{,}14{,}9{,}5)]{,}$\\
$[(0{,}7{,}1{,}14{,}10){,}(2{,}5{,}8{,}13{,}6){,}(3{,}11{,}4{,}9{,}12)],[(0{,}8{,}1{,}5{,}12){,}(2{,}3{,}10{,}14{,}11){,}(4{,}6{,}9{,}7{,}13)]{,}$\\
$[(0{,}9{,}13{,}12{,}5){,}(1{,}11{,}10{,}8{,}4){,}(2{,}7{,}14{,}6{,}3)],[(0{,}10{,}5{,}1{,}3){,}(9{,}2{,}8{,}11{,}12){,}(4{,}13{,}7{,}6{,}14)]{,}$\\
$[(0{,}11{,}14{,}2{,}13){,}(1{,}12{,}6{,}8{,}10){,}(3{,}5{,}7{,}9{,}4)],[(0{,}12{,}8{,}7{,}4){,}(1{,}10{,}2{,}14{,}3){,}(5{,}13{,}9{,}11{,}6)]{,}$\\
$[(0{,}13{,}5{,}10{,}9){,}(1{,}7{,}3{,}6{,}11){,}(2{,}12{,}4{,}14{,}8)],[(0{,}14{,}1{,}9{,}6){,}(2{,}4{,}11{,}7{,}5){,}(3{,}8{,}12{,}10{,}13)]$
\item HWP$^*(15;3^{3},5^{11})${,}\\
$[(0{,}1{,}2){,}(3{,}4{,}5){,}(6{,}7{,}8){,}(9{,}10{,}11){,}(12{,}13{,}14)],[(0{,}2{,}4){,}(6{,}8{,}10){,}(12{,}14{,}1){,}(3{,}5{,}7){,}(9{,}11{,}13)],$\\
$[(0{,}3{,}9){,}(1{,}8{,}2){,}(4{,}14{,}7){,}(5{,}10{,}13){,}(6{,}12{,}11)],[(0{,}4{,}2{,}9{,}14){,}(1{,}6{,}10{,}7{,}12){,}(3{,}13{,}8{,}5{,}11)],$\\
$[(0{,}5{,}14{,}13{,}1){,}(2{,}6{,}4{,}10{,}12){,}(3{,}7{,}11{,}8{,}9)],[(0{,}6{,}1{,}13{,}11){,}(2{,}10{,}3{,}12{,}7){,}(4{,}8{,}14{,}9{,}5)],$\\
$[(0{,}7{,}1{,}14{,}10){,}(2{,}5{,}8{,}13{,}6){,}(3{,}11{,}4{,}9{,}12)],[(0{,}8{,}1{,}5{,}12){,}(2{,}3{,}10{,}14{,}11){,}(4{,}6{,}9{,}7{,}13)],$\\
$[(0{,}9{,}13{,}12{,}5){,}(1{,}11{,}10{,}8{,}4){,}(2{,}7{,}14{,}6{,}3)],[(0{,}10{,}5{,}1{,}3){,}(9{,}2{,}8{,}11{,}12){,}(4{,}13{,}7{,}6{,}14)],$\\
$[(0{,}11{,}14{,}2{,}13){,}(1{,}7{,}10{,}4{,}3){,}(5{,}9{,}8{,}12{,}6)],[(0{,}12{,}8{,}3{,}6){,}(1{,}4{,}11{,}7{,}9){,}(2{,}14{,}5{,}13{,}10)],$\\
$[(0{,}13{,}3{,}14{,}8){,}(1{,}10{,}9{,}6{,}11){,}(2{,}12{,}4{,}7{,}5)],[(0{,}14{,}3{,}8{,}7){,}(1{,}9{,}4{,}12{,}10){,}(2{,}11{,}5{,}6{,}13)]$
\item HWP$^*(15;3^{5},5^{9})${,}\\
$[(0{,}1{,}2){,}(3{,}4{,}5){,}(6{,}7{,}8){,}(9{,}10{,}11){,}(12{,}13{,}14)],[(0{,}2{,}4){,}(6{,}8{,}1(0){,}(12{,}14{,}1){,}(3{,}5{,}7){,}(9{,}11{,}13)],$\\
$[(0{,}3{,}9){,}(1{,}8{,}2){,}(4{,}14{,}7){,}(5{,}10{,}13){,}(6{,}12{,}11)],[(0{,}4{,}7){,}(1{,}9{,}3){,}(2{,}14{,}13){,}(5{,}11{,}8){,}(10{,}12{,}6)],$\\
$[(0{,}5{,}13){,}(1{,}4{,}10){,}(2{,}9{,}14){,}(3{,}6{,}11){,}(7{,}12{,}8)],[(0{,}6{,}1{,}13{,}11){,}(2{,}10{,}3{,}12{,}7){,}(4{,}8{,}14{,}9{,}5)],$\\
$[(0{,}7{,}1{,}14{,}10){,}(2{,}5{,}8{,}13{,}6){,}(3{,}11{,}4{,}9{,}12)],[(0{,}8{,}1{,}5{,}12){,}(2{,}3{,}10{,}14{,}11){,}(4{,}6{,}9{,}7{,}13)],$\\
$[(0{,}9{,}13{,}12{,}5){,}(1{,}11{,}10{,}8{,}4){,}(2{,}7{,}14{,}6{,}3)],[(0{,}10{,}5{,}1{,}3){,}(9{,}2{,}8{,}11{,}12){,}(4{,}13{,}7{,}6{,}14)],$\\
$[(0{,}11{,}7{,}5{,}14){,}(1{,}10{,}4{,}2{,}12){,}(3{,}8{,}9{,}6{,}13)],[(0{,}12{,}2{,}13{,}1){,}(3{,}7{,}10{,}9{,}8){,}(4{,}11{,}14{,}5{,}6)],$\\
$[(0{,}13{,}10{,}2{,}6){,}(1{,}7{,}11{,}5{,}9){,}(3{,}14{,}8{,}12{,}4)],[(0{,}14{,}3{,}13{,}8){,}(1{,}6{,}5{,}2{,}11){,}(4{,}12{,}10{,}7{,}9)]$
\item HWP$^*(15;3^{7},5^{7})${,}\\
$[(0{,}1{,}2){,}(3{,}4{,}5){,}(6{,}7{,}8){,}(9{,}10{,}11){,}(12{,}13{,}14)],[(0{,}2{,}4){,}(6{,}8{,}10){,}(12{,}14{,}1){,}(3{,}5{,}7){,}(9{,}11{,}13)],$\\
$[(0{,}3{,}9){,}(1{,}8{,}2){,}(4{,}14{,}7){,}(5{,}10{,}13){,}(6{,}12{,}11)],[(0{,}4{,}7){,}(1{,}9{,}3){,}(2{,}14{,}13){,}(5{,}11{,}8){,}(10{,}12{,}6)],$\\
$[(0{,}5{,}13){,}(1{,}4{,}10){,}(2{,}9{,}14){,}(3{,}6{,}11){,}(7{,}12{,}8)],[(0{,}6{,}1){,}(2{,}13{,}10){,}(3{,}8{,}14){,}(4{,}11{,}5){,}(7{,}9{,}12)],$\\
$[(0{,}7{,}11){,}(2{,}6{,}5){,}(1{,}14{,}9){,}(4{,}12{,}10){,}(3{,}13{,}8)],[(0{,}8{,}1{,}5{,}12){,}(2{,}3{,}10{,}14{,}11){,}(4{,}6{,}9{,}7{,}13)],$\\
$[(0{,}9{,}13{,}12{,}5){,}(1{,}11{,}10{,}8{,}4){,}(2{,}7{,}14{,}6{,}3)],[(0{,}10{,}5{,}1{,}3){,}(9{,}2{,}8{,}11{,}12){,}(4{,}13{,}7{,}6{,}14)],$\\
$[(0{,}11{,}1{,}13{,}6){,}(2{,}5{,}14{,}8{,}12){,}(3{,}7{,}10{,}9{,}4)],[(0{,}12{,}3{,}14{,}10){,}(1{,}7{,}5{,}8{,}13){,}(2{,}11{,}4{,}9{,}6)],$\\
$[(0{,}13{,}3{,}11{,}14){,}(1{,}10{,}7{,}2{,}12){,}(4{,}8{,}9{,}5{,}6)],[(0{,}14{,}5{,}9{,}8){,}(1{,}6{,}13{,}11{,}7){,}(2{,}10{,}3{,}12{,}4)]$
\item HWP$^*(15;3^{9},5^{5})${,}\\
$[(0{,}1{,}3){,}(2{,}6{,}7){,}(4{,}5{,}8){,}(9{,}13{,}12){,}(10{,}11{,}14)],[(0{,}2{,}10){,}(1{,}6{,}8){,}(3{,}13{,}7){,}(4{,}12{,}14){,}(5{,}11{,}9)],$\\
$[(0{,}7{,}13){,}(1{,}8{,}6){,}(2{,}9{,}12){,}(3{,}10{,}14){,}(4{,}11{,}5)],[(0{,}9{,}8){,}(1{,}2{,}11){,}(3{,}14{,}12){,}(4{,}13{,}10){,}(5{,}7{,}6)],$\\
$[(0{,}10{,}5){,}(1{,}4{,}3){,}(2{,}13{,}8){,}(6{,}14{,}9){,}(7{,}11{,}12)],[(0{,}11{,}6){,}(1{,}12{,}13){,}(2{,}8{,}14){,}(3{,}4{,}9){,}(5{,}10{,}7)],$\\
$[(0{,}12{,}4){,}(1{,}14{,}5){,}(2{,}3{,}6){,}(7{,}10{,}8){,}(9{,}11{,}13)],[(0{,}13{,}14){,}(1{,}5{,}12){,}(2{,}7{,}4){,}(3{,}8{,}11){,}(6{,}9{,}10)],$\\
$[(0{,}14{,}11){,}(1{,}13{,}4){,}(2{,}12{,}10){,}(3{,}5{,}6){,}(7{,}8{,}9)],[(0{,}3{,}9{,}1{,}7){,}(2{,}4{,}14{,}8{,}5){,}(10{,}13{,}6{,}12{,}11)],$\\
$[(0{,}4{,}7{,}1{,}9){,}(3{,}2{,}14{,}13{,}5){,}(11{,}8{,}10{,}12{,}6)],[(0{,}5{,}13{,}2{,}1){,}(4{,}10{,}9{,}14{,}6){,}(3{,}11{,}7{,}12{,}8)],$\\
$[(0{,}6{,}13{,}11{,}2){,}(1{,}10{,}3{,}7{,}14){,}(4{,}8{,}12{,}5{,}9)],[(0{,}8{,}13{,}3{,}12){,}(1{,}11{,}4{,}6{,}10){,}(2{,}5{,}14{,}7{,}9)]$
\end{enumerate}

\begin{thebibliography}{amsplain}
\bibitem{Alspach2003} B. Alspach, H. Gavlas, M. Sajna, and H. Verrall, {\em Cycle decompositions IV: complete directed graphs and fixed length directed cycles}, J. Comb. Theory Ser. A. {\bf 103(1)} (2003), 165-208.

\bibitem{Alspach1985} B. Alspach, R. Haggkvist, {\em Some observations on the Oberwolfach problem}, J. Graph Theory, {\bf 9} (1985), 177-187.

\bibitem{Alspach1989} B. Alspach, P. J. Schellenberg, D. R. Stinson and D. Wagner, {\em The Oberwolfach problem and factors of uniform odd length cycles}, J. Combin. Theory Ser. A, {\bf 52} (1989), 20-43.

\bibitem{Adams2002}
P. Adams, E. J. Billington, D. E. Bryant, and S. I. El-Zanati, {\em On the Hamilton-Waterloo problem}, Graphs Combin. {\bf 18} (2002), 31-51.

\bibitem{Adams2006}
P. Adams and D. Bryant, {\em Two-factorisations of complete graphs of orders fifteen and
seventeen}, Australasian J. of Combinatorics, {\bf 35} (2006), 113-118.

\bibitem{D.Bryant2011} 
D. Bryant, P. Danziger, {\em On bipartite $2$-factorizations of $K_{n}-I$ and the Oberwolfach problem}, J. Graph Theory {\bf 68(1)} (2011), 22-37.

\bibitem{D.Bryant2013} 
D. Bryant, P. Danziger, and M. Dean, {\em On the Hamilton-Waterloo Problem for Bipartite 2-Factors}, J. Comb. Des. {\bf 21(2)} (2013), 60-80.

\bibitem{Bermond1979}
J. C. Bermond, A. Germa, and D. Sotteau, {\em Resolvable decomposition of $K_{n}^{*}$}, J. Comb. Theory Ser. A. {\bf 26(2)} (1979), 179-185.

\bibitem{Bennett1990}
F. E. Bennett, X. Zhang, {\em Resolvable Mendelsohn designs with block size 4}, Aequationes Math. {\bf 40(1)} (1990), 248-260.

\bibitem{Sajna2018}
A. Burgess, N. Francetic, and M. Sajna, {\em On the directed Oberwolfach Problem with equal cycle lengths: the odd case}, Australas. J. Comb. {\bf 71(2)} (2018), 272-292.

\bibitem{Sajna2014}
A. Burgess, M. Sajna, {\em On the directed Oberwolfach Problem with equal cycle lengths}, Electron. J. Comb. {\bf 21(1)} (2014), 1-15.

\bibitem{Burgess2017}
A. Burgess, P. Danziger, and T. Traetta, {\em On the Hamilton-Waterloo problem with odd orders}, J. Comb. Des. {\bf 25(6)} (2017), 258-287.

\bibitem{Burgess2018}
A. Burgess, P. Danziger, and T. Traetta, {\em On the Hamilton-Waterloo problem with odd cycle lengths}, J. Comb. Des. {\bf 26(2)} (2018), 51-83.

\bibitem{Bonvicini2017} 
S. Bonvicini, M. Buratti, {\em Octahedral, dicyclic and special linear solutions of some Hamilton-Waterloo problems}, Ars Math. Contemp. {\bf 14(1)} (2017), 1-14.

\bibitem{Danziger2009} 
P. Danziger, G. Quattrocchi, and B. Stevens, {\em The Hamilton-Waterloo problem for cycle sizes 3 and 4}, J. Comb. Des., {\bf 17(4)} (2009), 342-352.

\bibitem{Ringel1971} R. K. Guy, {\em Unsolved  combinatorial  problems,  In:  Proceedings  of  the Conference  on  Combinatorial  Mathematics  and  Its  Applications} ,  Oxford, 1967 (D. J. A. Welsh, Ed.), Academic Press, New York, 1971.

\bibitem{Deza2010}
A. Deza, F. Franek, W. Hua, M. Meszka, A. Rosa, {\em Solutions to the Oberwolfach problem for orders 18 to 40}, Journal of Combinatorial Mathematics and Combinatorial Computing, {\bf 74} (2010), 95102.

\bibitem{Haggkvist1985} R. Haggkvist, {\em A lemma on cycle decompositions}, North-Holland Mathematics Studies {\bf 115} (1985), 227-232.

\bibitem{Hoffman1991} D.G. Hoffman, P. J. Schellenberg, {\em The existence of $C_k-$factorizations of $K_{2n}-F$}, Discrete Math. {\bf 97} (1991), 243-250.

\bibitem{ProductGraphs} 
W. Imrich, S. Klavzar {\em Product graphs: Structure and Recognition}, John Wiley and Sons Incorporated, New York, 2000.

\bibitem{Keranen} 
M. Keranen, S. \"Ozkan, {\em The Hamilton-Waterloo problem with $4$-cycles and a single factor of $n$-cycles}, Graphs Combin. {\bf 29 } (2013), 1827–1837.

\bibitem{Liu2003} 
J. Liu, {\em The equipartite Oberwolfach problem with uniform tables}, J. Comb. Theory Ser. A. {\bf 101} (2003), 20–34.

\bibitem{Sajna2020}
E. Shabani, M. Sajna, {\em On the Directed Oberwolfach Problem with variable cycle lengths}, 2020, arXiv preprint arXiv:2009.08731.

\bibitem{Odabas2016}
U. Odaba\c{s}{\i}, S. \"Ozkan, {\em The Hamilton-Waterloo problem with $C_{4}$ and $C_{m}$ factors}, Discrete Math. {\bf 339(1)} (2016), 263-269.
\end{thebibliography}
\end{document}